\documentclass[a4paper,11pt,reqno]{amsart}
\usepackage{mathrsfs,frcursive}
\usepackage{mathtools}
\usepackage[usenames, dvipsnames]{color}

\usepackage{hyperref}
\usepackage{graphicx}
\RequirePackage{amsthm,amsmath,amsfonts,amssymb}
\RequirePackage[numbers,sort&compress]{natbib}
\RequirePackage{nicefrac}
\usepackage[all]{xy}
\newcommand{\R}{\mathbb{R}}
\newcommand{\Rm}{\mathbb{R}^+}
\newcommand{\C}{\mathbb{C}}
\newcommand{\N}{\mathbb{N}}
\newcommand{\Z}{\mathbb{Z}}
\newcommand{\h}{{H}}
\newcommand{\bb}{\beta}
\newcommand{\X}{\mathcal{X}}

\newcommand{\F}{\mathcal{F}}
\newcommand{\D}{\mathcal{D}}
\newcommand{\iF}{\mathcal{F}^{-1}}
\newcommand{\vl}{v^{\lambda}}
\newcommand{\tv}{\tilde{v}}
\newcommand{\tu}{\tilde{u}_0}
\newcommand{\supp}{\text{supp}}
\newcommand{\bv}{\Bar{v}}
\usepackage[none]{hyphenat}
\numberwithin{equation}{section}

\newtheorem{theorem}{Theorem}[section]
\newtheorem{lemma}[theorem]{Lemma}
\newtheorem{claim}[theorem]{Claim}
\newtheorem{proposition}[theorem]{Proposition}

\newtheorem{definition}[theorem]{Definition}
\newtheorem{remark}[theorem]{Remark}
\newcommand{\beq}{\begin{equation}}
\newcommand{\eeq}{\end{equation}}
\newcommand{\be}{\begin{eqnarray*}}
\newcommand{\ee}{\end{eqnarray*}}
\newcommand{\beqn}{\begin{equation*}}
\newcommand{\eeqn}{\end{equation*}}
\newcommand{\bc}{\begin{cases}}
\newcommand{\ec}{\end{cases}}

\newcommand{\SLT}{\square}

\title[$\mathbf{L^{2}}-$ Well-posedness and Bounded Controllability of KdV-B equation]{$L^{2
}$ Well-posedness and Bounded controllability for the Korteweg-de Vries-Burger equation in a half-plane \\ }

\author[Esquivel]{Liliana Esquivel}
\address{Departament of Mathematical Science, Universidad de Puerto Rico Recinto Mayag\"uez, 259 Av. Alfonso Valdés Cobián, Mayagüez, Puerto Rico}
\email{liliana.esquivel@upr.edu} 

\author[Rivas]{Ivonne Rivas}
\address{Universidad del Valle, Departamento de Matem\'aticas, Calle 13 No. 100 - 00, Ciudadela Universitaria Mel\'endez, Cali, Colombia.} 
\email{ivonne.rivas@correounivalle.edu.co}
\date{}

\begin{document}

\begin{abstract}
In this paper, the initial boundary value problem of the Korteweg-de Vries Burger equation on the negative half-plane is analyzed.  Initially, the well-posedness on $H^s(\R^-)$ for $s\geq 0$ of the IBVP is established to concentrate on the  $L^2(\R^-)$ controllability problem when the controls are in the Dirichlet and Newmann conditions at $x=0$.	\end{abstract}
	
	\maketitle
	
\section{Introduction}

 This paper addresses a dispersive problem in an unbounded domain.  The Korteweg-de Vries-Burger  (KdV-B) equation on the negative half-line. The initial boundary value problems (IBVP)  for the Generalized KdV-B equations on the  half line are\\

 On the positive half-line
 \begin{equation}\label{1.1.0}  
        \begin{cases}
        u_{t}+u_{xxx}-u_{xx}=\frac{1}{2}(u^2)_x,\ &t>0, \ x>0, \\
		u(x,0)=u_{0}(x), &x>0,\\
		u(0,t)=h(t), &t > 0,   
  \end{cases}
	\end{equation}
 which requires just one boundary condition to be well-defined.\\

 On the negative half-line
     	\begin{equation} \label{1.2} \begin{cases}u_{t}+u_{xxx}-u_{xx}=\frac{1}{2}(u^2)_x,& t>0, \ x<0, \\
		u(x,0)=u_{0}(x),& x<0,\\
		u(0,t)=h(t), \ u_x(0,t)=g(t),  &t > 0, 
  \end{cases}
\end{equation}
 than requires two boundary conditions to be well-defined. 
 It is interesting to observe, that the IBVPs   \eqref{1.1.0} and \eqref{1.2} can be summarized on one problem, considering the parameter $a= 0,1$, as it follows
 	\begin{equation}\label{1.1} \begin{cases}u_{t}+(-1)^a u_{xxx}-u_{xx}=\frac{1}{2}(u^2)_x,& t>0, \ x>0, \\
 		u(x,0)=u_{0}(x), &x>0,\\
 		u(0,t)=h(t),  &t > 0,  \textnormal{ }\\
 	 u_x(0,t)=g(t),  &t>0, \quad \text{ only in the case }a =1, \end{cases}
 \end{equation}
which allows us to study just one of the problems.\\

The KdV-B equation was derived by Sun and Gardner, see \cite{SuGardner}, as a dispersion, viscosity, and nonlinear advection model. This equation also models long wave propagation in shallow water in \cite{RS1}, propagation of waves in elastic tube stuffed with a viscous fluid in \cite{RS2}, weakly nonlinear plasma waves with dissipative effects, see \cite{Hu}, among other applications.  Mathematically, the KdV-B equation is a long wavelength approximation of the impact of the non-linearity and the dispersion term.\\

If only the higher order terms are considered in the linear operator the equation becomes the very well study KdV equation $u_{t}+ u_{xxx}=uu_x,$ derived by Korteweg and Vries in \cite{ KDV} as an equation that models surface waves with small amplitude, long surface propagating in a shallow channel of water. However,
if it is considered only the lower terms the Burger equation is archived $u_t-u_{xx}=uu_x$, which is proposed initially by  Bateman, see \cite{Bateman} and later studied by Burger in \cite{Burger}  as a model of the turbulence and approximate theory of flow through a shock wave traveling in a viscous fluid in \cite{Gupta1,Wazwaz}.  When diffusion dominates dispersion, the numerical solutions of the KdV-B tend to behave like Burgers equation solutions, that is the steady-state solutions of the KdV-B equation are monotonic shocks. Moreover, when dispersion dominates, the KdV equation behavior is observed and the shocks are oscillatory, see \cite{Gupta2}.\\

The last result known by the authors for the well-posedness problem of the KdV-B equation \eqref{1.1.0} was made by Bona, Sun, and Zhang in \cite{BSZ08}, when the KdV-B has a transport term, $u_x$, showing that it is local well-posedness in $H^s(\R^+)$  for $s>-1$ and $s\neq 3m+\frac{1}{2}$ with $m=0,1,2,...$ when the data $u_0\in \h^{s}(\R^+) $ and the boundary condition is considered at $u(0,t)=h(t)\in \h^{\frac{s+1}{3}}(\R^+)$.  The results are based on the combination of the dispersion of the third-order derivative of the KdV equation and the dissipation of the Burger, in the negative second derivative term.  However the KdV-B equation in $\R$ has been studied by Kappeler and Topalov \cite{KT}, showing that it is globally well-posed in $H^s(\R)$ for $s\ge -1$, using inverse scattering approach, and simultaneously  Molinet and Ribaud \cite{MoRi} showing that the KdV-B in $\R$ is well-posed in $H^s(\R)$ for $s>-1$, using Bourgain and Kenig, Ponce and Vega approach, and ill-posed for $s<-1$, in the sense that the flow-map
$u_0 \to u(t)$ is not of class ${C}^2$-map from $\h^{\mathrm{s}}(\mathbb{R})$ to ${C}\left([0, {T}] ; \h^{\mathrm{s}}(\mathbb{R})\right)$.\\

The impact of the boundary condition at $x=0$ can be seen when  the equation \eqref{1.1} with $a=1$ is multiplied by $u$, assuming that $u$ is a smooth solution decaying rapidly to zero at infinity i.e. 
$$\frac{d}{dt} \int_0^{\infty} u^2(x,t)dx = -\frac{1}{2}g^2(t)+ h(t)\left( u_{xx}(0)+g(t)\right) - \int_0^{\infty}(u_x)^2dx, $$
from where, if $h\equiv 0$ the dissipation effect arises naturally and the dissipative effect of the KdV term is contributing subtly, together with the Burgers term. In this paper, it is required to find a full characterization of the solution of IBVP \eqref{1.2}, particularly to be able to analyze the regularity of the solution and the trace properties that are important in the study of the controllability.\\

Chapouly in \cite{Chapouly} has studied the exact controllability of the Burgers equation on an interval in the non-viscous and viscous cases when three controls are in play on the forcing term, and in each in the boundary terms, i.e $y(t,0)$ and $y(t,L)$. Marbach in \cite{Marbach} studied the exact controllability  of the systems     
\begin{gather*}
    \begin{cases}
    y_t-y_{xx}+yy_x=f(x,t)\chi_{[a,b]}\\
    y(0,t)=0\\
    y(1,t)=0
    \end{cases} \quad \text{and} \quad
    \begin{cases}
    y_t-y_{xx}+yy_x=0\\
    y(0,t)=0\\
    y(1,t)=h_1(t)
    \end{cases}
\end{gather*}

Chen \cite{Chen} showed time optimal internal controllability to zero in $[0,L]$ with homogeneous boundary conditions.  In the case of unbounded domains, the challenge falls over the property of compactness since it is important in the proof of controllability.  Gallego in \cite{Gallego} has studied the controllability of the KdV-Burgers in $\R$, showing the internal controllability when the solutions are assumed to be $L^2_{Loc}(\R)$ following the ideas of \cite{Rosierub}. However, a difference between Gallego's works and this paper is that we are considering $\mathbb R^-$ with two boundary conditions, and therefore more analysis in the boundary term, the Carleman inequality gives a different estimate, and the approaches used in the operator are different.  In this work, we obtain the natural bounded spaces for the problem addressed. \\

The paper is distributed as follows: in section 2, some helpful notation is introduced, and the main results are stated. In section 3,  the well-posedness of the IBVP \eqref{1.2} is addressed. The problem is divided into three different types of linear problems, and for each one, the respective estimates are computed to conclude with section 4, with the wellposedness of the IBVP of \eqref{1.2}.   In section 5, Carleman's estimate is shown.  In section 6, the proof of a negative result of controllability for the bounded solution of the KdV-B is proved, and the positive result of boundary controllability is shown, by using an approximation theorem together with particular claims that are shown in the appendix section.  

 \section{Notation and Main result}\label{Notation}
  
\noindent Some necessary notations and definitions of function spaces used in the formulating are introduced for our main result and further analysis. For $t\in \mathbb{R},$ we denote by $\langle  t \rangle:=\sqrt{1+|t|^2}$ and  $\{t\}:=\frac{|t|}{\langle t \rangle }$.  $A\lesssim_s B$ means that there exists a positive constant $c,$ such that $c$ does not depend on fundamental quantities on $A$ and $B,$ but depends on $s$ such that $A\leq c(s)B$, however when there is not a direct dependence the constant is denoted by $c$.\\
	
Henceforth, $\mathcal{S}(\R)$ denotes the space of rapidly decreasing functions.  We denote the usual Laplace transform,  $\mathcal{L}$, and its inverse, $\mathcal{L}^{-1},$ and  the integral operators are 
$$\mathcal{L}\phi (\xi)\equiv\widetilde{\phi }\left( \xi \right) :=\int\limits_{{0}}^{\infty }e^{-x\xi}\phi \left( x\right) dx \quad \text{and} \quad \mathcal{L}^{-1}\phi(x) =\frac{1}{2\pi i}\int\limits_{i\mathbf{R}}e^{ix\xi }\widetilde{\phi }\left( \xi \right) d\xi.	$$
The Fourier transform, $\mathcal{F}$  and its inverse, $\mathcal{F}^{-1}$,  are defined by 
$$\mathcal{F}(\phi)(p)= \widehat{\phi}(p):=\frac{1}{\sqrt{2\pi}}\int_{-\infty}^{\infty}e^{-ix p}\phi(x)dx \quad \text{and} \quad  \mathcal{F}^{-1}(\phi)(p)= {\phi}^{\vee}(x):=\frac{1}{\sqrt{2\pi}}\int_{\infty}^{\infty}e^{ix p}\widehat{\phi}(p)dp.$$

For $s\in \R $, the space  $\h^s(\R)$ is the set of distributions $f$ satisfying $(1+|\xi|)^s\widehat{f}(\xi)\in {L}^2_\xi(\R)$. \\
    
The space $\h^s_0(\R)$, is the closure of the class of functions in $\h^s(\mathbb{R})$ whose support lies in $\Rm.$ The homogeneous space
$\dot{\h^s}(\R)$  consisted of distributions satisfying  $|\xi|^s\widehat{f}(\xi)\in  L^2_\xi(\R)$.\\  

For  $l\in \R$, such that  $0<l<L$ and $t_1,t_2 \in\R$, with $0<t_1<t_2<T$, we consider the  domains 
\begin{equation}\label{square}
\square_L:= (-L,L)\times(0,T),
\end{equation}
for $\delta>0$  we define
\begin{equation}\label{squaredelta}
\square_\delta:=(-l,l)\times (t_1-2\delta, t_2+2\delta), 
\end{equation}
when it is necessary to emphasize the space size, we will write $\square_{\delta,l}$.\\

The $L^2_{\beta}-$weighted space, for $\bb >0$, is defined by
 $$ L^2_{\bb}=\left\{ u:(0,\infty)\to \R \quad \text{such that}\quad \int_0^\infty u^2(x)e^{-2\bb x}dx<\infty\right\}.$$
For $s\in\R$, $0\le b\le 1$ and any function $ w(x,t):\R^2  \to \R$ set 
		$$\nu_{s,b}^2(w):=\int_{-\infty}^{\infty} \int_{-\infty}^{\infty} <i(\tau + \xi^3)+ \xi^2>^{2b}<\xi>^{2s} |\widehat{\widehat{w}}(\xi,\tau)|^2d\xi d\tau \le \infty.$$
The spaces  $X_{s,b}$ be the completion of the space of all the functions 
	\begin{equation}\label{xsb}
\Big\{ w(x,t):\R^2 \to \R \, \quad \text{such that} \quad  \|w\|_{X_{s,b}}:= \nu_{s,b}(w) \le \infty \Big\}
	\end{equation}
 and  the space $\chi_{s,b}:= C(\R, \h^s(\R)) \cap X_{s,b}$ with the norm 
 $$\|w\|_{\chi_{s,b}}:= \left( \sup_{t\in\R} \|w(.,t)\|_{\h^s(\R)}^2 + \|w\|_{X_{s,b}}^2 \right)^{\frac12}.$$

Moreover,  for $s\in \mathbb R$, $0 \leq b \leq 1$ we consider the  auxiliary functional spaces $\mathcal{V}_{s, b}$, $\mathcal{Y}_{s, b}$, with the norm
%$$
%\|w\|_{\mathcal{Y}_{s, b}}=\left(\|w\|_{\mathcal{X}_{s, b}}^2+\sup _{x \in \mathbb{R}}\left(\left\|w(x, \cdot)\right\|_{H_t^{\frac{s}{3}}( \mathbb R)}^2\right)\right)^{\frac{1}{2}} ,\textcolor{blue}{Para que se necesita}
%$$
$$
\|w\|_{\mathcal{V}_{s, b}}=\left(\|w\|_{\mathcal{X}_{s, b}}^2+\sup _{x \in \mathbb{R}}\left(\left\|w(x, \cdot)\right\|_{H_t^{\frac{s+1}{3}}( \mathbb R)}^2+ \left\|w_x(x, \cdot)\right\|_{H_t^{\frac{s}{3}}( \mathbb R)}^2\right)\right)^{\frac{1}{2}} .
$$

Let $\Omega\subset \mathbb{R}$ be an interval, the restriction norm  $\h^s(\Omega)$ is     
    $$\|f\|_{\h^s(\Omega)}=\inf_{\Omega} \|F\|_{\h^{s}(\R)}$$ 
and the  space $\chi_{s, b}(\mathbb R^{+} \times \Omega)$ is define as follows:
$$\chi_{s, b}\left( \Omega\times \mathbb R^{+}\right)=\chi_{s, b}\Big|_{\mathbb R^+\times \Omega},$$
with the norm
$$
\|u\|_{\chi_{s, b}\left(\mathbb R^+ \times \Omega\right)} := \inf _{\mathfrak{u} \in X_{s, b}}\left\{\|\mathfrak{u}\|_{X_{s, b}}: \mathfrak{u}(x, t)=u(x, t) \text { on } \mathbb{R}^{+} \times \Omega\right\} .
$$
For $T \geq 0$, the restricted Bourgain space on time $X^{s, b}_T$, endowed with the norm
$$
\|w\|_{X_{T}^{s,b}}=\inf_{\mathrm{w} \in {X}^{ s,b}}\left\{\|\mathrm{w}\|_{{X}^{ s,b}}, w(\mathrm{t})=\mathrm{w}(\mathrm{t}) \text { on }[0, T]\right\} .
$$

The notion of solution used in this paper follows the definition described by  Bona, Sun, and Zhang in \cite{BSZ04} 
\begin{definition}[Mild solution]
    For $s\in \mathbb R$ and $T>0$ given. $u(x,t)$ is a mild solution of \eqref{1.1} on $[0,T]$ if exists a sequence $u_n \in C([0,T];H^3(\R^+))\cap  C^1([0,T];L^2(\R^+))$ such that
    \begin{enumerate}
        \item $u_n$ solves \eqref{1.1} in $L^2(\R^+)$ with $t\in [0,T]$
        \item  $\lim_{n\to \infty} \|u_n-u\|_{C([0,T];H^s(\R^+))}=0$
        \item $\lim_{n\to \infty} \|u_n(x,0)-u_0(x)\|_{H^s(\R^+)}=0$, $\lim_{n\to \infty} \|u_n(0,t)-h(t)\|_{H^{\frac{s+1}{3}}([0,T])}=0$ \ and \ $\lim_{n\to \infty} \|\partial_x u_n(0,t)-g(t)\|_{H^{\frac{s}{3}}([0,T])}=0$.
    \end{enumerate}
\end{definition}
together with the compatibility conditions defined in \cite{BSZ04}. \\

The  main results of the paper are
\begin{itemize}
\item Local well-posedness of the KdV-B equation \eqref{1.2} in $H^s(\R^-)$ for $s>0$.
\begin{theorem}\label{wellposedness}

For $T>0$, consider $s\ge0$ such that  $s \neq 3m+\frac{1}{2}, \ m=0,1,2, \ldots$. , there exists a sufficiently small $r=r(s,T),$ such that if  
$$\|u_0\|_{\h^s\left(\mathbb R^{+}\right)}+\|h\|_{\h^{\frac{s+1}{3}}(0,T)}+\|g\|_{\h^{\frac{s}{3}}(0,T)}\le r,$$   
the IBVP \eqref{1.2} has a unique solution $u \in C\left([0, T] ; \h^s\left(\mathbb R^{+}\right)\right)$. Furthermore, the solution $u$ depends on Lipschitz continuously on $(u_0, h, g)$ in their respective spaces.
\end{theorem}
\item No controllability of the KdV-B equation for a $L^2-$bounded solutions.\\

\noindent
There is a lack of exact controllability for the linear KdV-B when the solutions have bounded energy, $u \in L^{\infty}(0,T;L^2(\Rm))$. In other words, the system \eqref{1.2} fails to be null controllable. As Rosier proved it in \cite[Theorem 1.2]{Rosierub} for the KdV equation,  the null-controllability cannot be reached if the solution is $L^2$ bounded.  The same effect still happens even if a damping term  $u_{xx}$ is added. Nevertheless, the proof is not straightforward because the operator does not have the same properties, it is not skew-adjoint or even adjoint.  
\begin{theorem}\label{nonullcontrol}
    Let $T>0$. Then there exists $u_0 \in L^2(\Rm)$ such that if $u\in L^{\infty}(0,T,L^2(\Rm))$ satisfying 
    \begin{equation}\label{nonull}
    \begin{cases}
        u_t - u_{xxx}-u_{xx}=0 & \text{in} \quad \D'((0,\infty) \times (0,T)),\\
        u(x,0)=u_0,
        \end{cases}
    \end{equation}
then $u(x,T)\neq 0$.
\end{theorem}

\item Controllability of the KdV-B equation for a $L^2_{Loc}-$solutions
	\begin{theorem}\label{control}
   Let $T>0$, $\bb >0$  and $\epsilon \in (0,T/2)$.   For $u_0\in L^2(0,\infty)$ and $u_T\in L^2_{\bb}$, there exists a function 
   $$u\in  L^2_{loc}([0,\infty) \times [0,T] ) \cap C([0,\epsilon];L^2[0,\infty)) \cap C([T,T-\epsilon];L^2_{\beta}[0,\infty))$$
   which solves 
   \begin{equation}\label{exact}
    \begin{cases}
        u_t - u_{xxx}-u_{xx}=0, & \text{in} \quad \D'((0,\infty)\times (0,T)),\\
        u(x,0)=u_0,
        \end{cases}
    \end{equation}
    satisfying $u(x,T)=u_T.$
\end{theorem}
\end{itemize}

 \section{Analysis of the linear Korteweg de Vries-Burgers equation}
 The characterization of the solution for the KdV-B equation is established in this section, to find expressions in an appropriate form to obtain estimates that help us show the equation is well-posedness in the desired spaces.
\subsection{Solution formulas for linear problems}
	\subsubsection{\bf Initial value problem in $\mathbb R$}\label{inR}
	Let us consider the linear IVP for KdV-B equation on the full line,	
	\begin{equation}\label{IVP}
	    \begin{cases}
	        u_{t}-u_{xxx}-u_{xx}=0,\ &t>0, \ x\in \R,\\
			u(x,0)=u_0, \ &x\in \R. 
		\end{cases}
	\end{equation}
	
	The solution of \eqref{IVP}  for  $u_0 \in  \h^s(\R)$ with $s\in \R$ can be written as 
 \begin{align}
     \label{WR}
	    u(x,t)&:= W_{\R}(t)u_0(x) \notag\\
     &= \iF\left[e^{-( i \xi^3 + \xi^2)t}\,\F[u_0](\xi) \right] 
	    = \frac{1}{\sqrt{2\pi}} \int_{-\infty}^{\infty} e^{-( i \xi^3 + \xi^2)t}e^{-i \xi x} \int_{-\infty}^{\infty} e^{i y\xi} u_0(y)dy d\xi, 
 \end{align}
 
in \cite{BSZ08}[Proposition 2.4] Bona, Sun and Zhang have studied and concluded several estimates for the IVP \eqref{IVP} we recall the results that are required for our proof: \\
\begin{proposition}\label{RlinearID}
Let $s\in \R$, $0<b\le 1$ and $0 < \delta  < \frac{1}{2}$ be given. For $u_0 \in C^{\infty}_0(\R)$ (smooth function with compact support):
\begin{enumerate}
    \item  There exists a constant $C:=C(s,b,\varphi)>0$ such that
    \begin{gather}
    \| \varphi(t) \,W_{ \R}(t) u_0 \|_{\chi_{s,b}} \leq C \|u_0 \|_{\h^{s}(\R)}
    \end{gather}
    
%\textcolor{blue}{ \item There exists $C_{\delta}>0$ such that for all $u\in \chi^{s,\frac{1}{2}}$ \begin{gather} \| \varphi(t) \, \int_0^t W_{ \R}(t-t') f(t')dt' \|_{\chi^{s,\frac{1}{2}}} \le C_{\delta} \|f \|_{X_{s,-\frac{1}{2}+\delta}} \end{gather}}
    \item There exists a constant $C:=C(s)>0$ such that
    \begin{gather}
    \sup_{x\in \R} \| W_{ \R} u_0 \|_{\h_t^{\frac{s}{3}}(\R)} \le C \|u_0 \|_{\h^{s}(\R)}
    \end{gather}
	
	\item For all $f_n\in  X_{s, -\frac{1}{2}+\delta}$, the mapping 
	$$t \to \int_0^t W_{\R}(t-t') f_n(t')dt' $$ 
	belong to  ${C}(\mathbb R^+,\h^{s+2\delta}(\R))$.  Moreover,  if $f_n \to 0$ in $ X_{s, -\frac{1}{2}+\delta}$ then 
	$$\left\|  \int_0^t W_{\R}(t-t') f_n(t')dt'\right \|_{\mathbf L^{\infty}(\R^+, \h^{s+2\delta}(\R))} \to 0 \quad \text{as} \quad n\to \infty.$$
	\end{enumerate}
	\end{proposition}

	\begin{proposition}\label{RlinearIDRegularity}{\cite[propositions 2.5]{BSZ08}}
	Let $-1 \le s \le 5$.  There exists a constant $C:=C(s)$ such that 
\begin{gather}
    \sup \limits_{x\in \mathbb R} \|W_{\R}(t) u_0\|_{H_t^{\frac{s+1}{3}(\R)}}\le C\|u_0 \|_{H^s(\R)},\\
    \sup  \limits_{x\in \mathbb R} \|\partial_x W_{\R}(t) u_0\|_{H_t^{\frac{s}{3}(\R)}}\le C\| u_0\|_{H^s(\R)},
\end{gather}
for any $u_0 \in H^s(\R)$.
	\end{proposition}

\begin{proposition}\label{RlinearForcing}{\cite[propositions 2.7]{BSZ08}}
For $-1\le s\le 2-3b$ with $0\le b < \frac{1}{2}$ and     $ f\in \mathbf{C}_0^{\infty}(\R)$, considering 
$$w(x,t)= \int_0^t W_{\R}(t-t')f(.,t')dt',$$
then, there exists $C:= C(b,s,\varphi)$ such that  
\begin{gather}
    \sup \| \varphi(.) w(x,.) \|_{H_t^{\frac{s+1}{3}}(\R)} \le C \| f\|_{X_{s,-b}},\\
    \sup \| \varphi(.) w_x(x,.) \|_{H_t^{\frac{s}{3}}(\R)} \le C \| f\|_{X_{s,-b}}.
\end{gather}

\end{proposition}
	
	{\bf Proof Propositions \eqref{RlinearID}, \eqref{RlinearIDRegularity}  and \eqref{RlinearForcing}}
	All the proofs follow the results on \cite{BSZ08}, the difference is in the symbol of the kdV-B equation \eqref{IVP} is $-i \xi^3 t -  \xi^2 t $ and the symbol of de KdV in \cite{BSZ08} is $i \xi^3 t - \xi^2 t $ when $\rho=0$. The change of sign in the cubic term can be addressed by a simple change $\xi \to -\xi$ when $\xi \in \R$ will be sufficient for the adaptation in the proof. $\square$
	
	\begin{remark}
	 For $u_0 \in H^s(\R)$ with $s\in \R$.  If $t\in [0,T]$ for any $a\in \R^+$ there is a hidden regularity given by 
	  \begin{gather}\| u\|_{ C(\R^+,H^{s+\frac{a}{2}}(\R))} \le C(t) \|u_0\|_{H^s(\R)} \end{gather}
	 with $C(t)>0$ for $t\ge a$. 	 
	 However, if $t\in [0,\infty)$ an extra regularity is not reached, therefore 
	 \begin{gather}
	     \| u\|_{C(\R^+,H^s(\R))} \le C \|u_0\|_{H^s(\R)}.
	 \end{gather}
	\end{remark}

 It is well-known, that the $X^{s,b}$ spaces with $b > \frac{1}{2}$
are well-suited for studying the IVP of dispersive equations. Nonetheless, for this IBVP the standard argument cannot be applied directly because the boundary operator requires $b\in \left(0, \frac{1}{2}\right)$. Therefore, the estimates in Bourgain spaces for $b< \frac{1}{2}$ are required.

\begin{proposition} \label{Bourgain estimate prop}
Let $s\in \R$ and $0<b, \ b'<
 \frac{1}{2}$ be given. For $\varphi \in C^{\infty}_0(\R)$ 
    \begin{equation}
   \left\|\psi(t) \int_0^t W_{\mathbb R}\left(t-\tau^{\prime}\right) w\left(\tau^{\prime}\right) d \tau^{\prime}\right\|_{X_{s,b}} \leq C\|w\|_{X_{s,-b'}}. \label{C1}
\end{equation}
\end{proposition}

 \begin{proposition}[Bilinear estimate]\label{bilinearestimate}
 Given $s\geq 0, \, 0<b,b'<\frac{1}{2}$ and $T>0$, there exist positive constants $0<\mu<1$, $C$ and $\delta$ such that
$$
\left\|\partial_x(u v)\right\|_{X_T^{s,b'}} \leqslant C T^\mu\|u\|_{X^{s,b}_T}\|v\|_{X_T^{s, b}}
$$
for any $u, v \in X_T^{s, b}$.
 \end{proposition}

%The proof of propositions \ref{Bourgain estimate prop} and \ref{bilinearestimate} can be found in Appendix \ref{Bourgain estimates} and \ref{Bilinear estimate}, these are based on ideas present in \cite{CarvajalEsquivel, MoRi}.\\
%\textcolor{red}{O podriamos colocar, para omitir ese apéndice, lo siguiente: }\textcolor{blue}
The proofs of the statements \ref{Bourgain estimate prop} and \ref{bilinearestimate} have been omitted as they do not introduce any new ideas. They can be derived by following the methods presented in \cite{CarvajalEsquivel,MoRi}.

	\subsubsection{\bf Boundary value problem with homogeneous initial data}
 	Consider the linear Korteweg de Vries-Burgers equation with non-homogenous boundary condition
	\begin{equation} \label{lineal} \begin{cases}u_{t}-u_{xxx}-u_{xx}=0,\ t>0, \ x>0,\\
			u(x,0)=0, \ x>0,\\
			u(0,t)=h(t), \ u_x(0,t)=g(t), \ t > 0. & \textnormal{ } \end{cases}
	\end{equation}
	Assume that $h,g\in S(\mathbb R)$.
	Let $\widetilde{u}$,  $\widetilde{h}$ and $\widetilde{g}$ denote the Laplace transform of $u$, $h$ and $g$ with respect to $t$, respectively. By applying the Laplace transform to both sides of the equation in \eqref{lineal}, one obtains
 $$
\begin{cases}
    \tau \widetilde{u}(x,\tau)-\widetilde{u}_{x x x}(x,\tau)-\widetilde{u}_{x x}(x,\tau)=0, \\
    \widetilde{u}(0,\tau)=\widetilde{h}(\tau),\\
     \widetilde{u}_x(0,\tau)=\widetilde{g}(\tau).
\end{cases}$$
As both $\widetilde{u}(x,\tau)$ and $\widetilde{u}_x(x,\tau)$ tend to zero as $x \rightarrow+\infty$, it is concluded that for any $\tau$ with $\operatorname{Re} \tau>0$,
$$
\widetilde{u}(x,\tau)=c_1(\tau)e^{r_1(\tau) x}+c_2(\tau)e^{r_2(\tau) x},
$$
where $r_1(\tau)$ and $r_{2}(\tau)$ are the  solutions of
\begin{gather}\label{raices}
\tau-r^3-r^2=0
\end{gather}
satisfying $\operatorname{Re} r_j(\tau)<0$,  $j=1,2$. 
From the initial conditions, direct calculations give us that 
$$c_1(\tau)=\frac{r_2(\tau)\widetilde{h}(\tau)-\widetilde{g}(\tau)}{r_2(\tau)-r_1(\tau)} \text{ and } \quad c_2(\tau)=-\frac{r_1(\tau)\widetilde{h}(\tau)-\widetilde{g}(\tau)}{r_2(\tau)-r_1(\tau)}.$$
Thus, for any fixed $\gamma>0$, one has the representation
\begin{align*}
   u(x,t)&=\frac{1}{2 \pi i} \int_{-i \infty+\gamma}^{i \infty+\gamma} e^{\tau t}\widetilde{u}(x,\tau)  d \tau  \\
   &=\frac{1}{2 \pi i} \int_{-i \infty+\gamma}^{i \infty+\gamma} e^{\tau t}\left(\frac{r_2(\tau)e^{r_1(\tau)x}-r_1(\tau)e^{r_2(\tau)x}}{r_2(\tau)-r_1(\tau)}\widetilde{h}(\tau) + \frac{e^{r_2(\tau)x} -e^{r_1(\tau)x}}{r_2(\tau)-r_1(\tau)}\widetilde{g}(\tau)\right)d\tau.
\end{align*}

Using the fact that the integrant is analytic for $\operatorname{Re}(\tau)\geq 0$, Cauchy theorem implies that
\begin{align*}
   u(x,t)=\frac{1}{2 \pi i} \int_{-i \infty}^{i \infty} e^{\tau t}\left(\frac{r_2(\tau)e^{r_1(\tau)x}-r_1(\tau)e^{r_2(\tau)x}}{r_2(\tau)-r_1(\tau)}\widetilde{h}(\tau) + \frac{e^{r_2(\tau)x} -e^{r_1(\tau)x}}{r_2(\tau)-r_1(\tau)}\widetilde{g}(\tau)\right)d\tau.
\end{align*}

 Note that this representation shows that $u$ is naturally defined for all values of $t\in (0,\infty)$, being very useful for analyzing the time regularity of the solutions. However, to estimate spatial regularity we need to extend $W_{D}(t)h$ and $W_{N}(t)g$ to all $x\in \mathbb{R}$.\\

A change of variables $\tau= i\lambda^3$ and straightforward calculation reveals that for $x, t\geq 0$, the solution is rewrite as
\begin{align}\label{Wb}
   u(x,t)=W_{bdr}(h,g)=W_{D}(t)h+W_N(t)g,
\end{align}
with
    \begin{align*}
    W_D(t)h&=\frac{3}{2 \pi } \int_{- \infty}^{\infty} e^{i\lambda^3 t}\left(\frac{r_2(i\lambda^3)e^{r_1(i\lambda^3)x}-r_1(i\lambda^3)e^{r_2(i\lambda^3)x}}{r_2(i\lambda^3)-r_1(i\lambda^3)}\right)\lambda^2\widetilde{h}(i\lambda^3)d\lambda, \\
    W_{N}(t)g&=\frac{3}{2 \pi } \int_{- \infty}^{\infty} e^{i\lambda^3 t}\left(\frac{e^{r_2(i\lambda^3)x} -e^{r_1(i\lambda^3)x}}{r_2(i\lambda^3)-r_1(i\lambda^3)} \right)\lambda^2\widetilde{g}(i\lambda^3)d\lambda.
\end{align*}

Since $r_j(i\lambda^3)=i\lambda +\mu_j(\lambda)$ where $|\mu_j(\lambda)|\sim |\lambda|^{\frac{1}{3}}$ and $\operatorname{Re}(\mu_j(\lambda))<
0$,  $j=1,2$, the Fourier transform in space can be considered on the expressions
\begin{align*}
    &W_D(t)h={3}\mathcal{F}_{\lambda\to x}^{-1}\left\{e^{i\lambda^3 t}\frac{r_2(i\lambda^3)e^{\mu_1(\lambda)x}-r_1(i\lambda^3)e^{\mu_2(\lambda)x}}{r_2(i\lambda^3)-r_1(i\lambda^3)}\lambda^2\widetilde{h}(i\lambda^3)\right\},\\
   &W_N(t)g:={3}\mathcal{F}_{\lambda\to x}^{-1}\left\{e^{i\lambda^3 t}\frac{e^{\mu_2(\lambda)x} -e^{\mu_1(\lambda)x}}{r_2(i\lambda^3)-r_1(i\lambda^3)} \lambda^2\widetilde{g}(i\lambda^3)\right\}.
\end{align*}
Thus, the extensions to $\R$ of the  operators  follow as
\begin{align}\label{12345}
    W_D(t)h=W_{\R}\, \mathfrak{h} \quad \text{and} \quad 
    W_N(t)g=W_{\R}\, \mathfrak{g},
\end{align}
where $\mathfrak{h}$ and $\mathfrak{g}$ are such that
\begin{align}
    &\mathcal{F}\{\mathfrak{h}\}={3}\left(\frac{r_2(i\lambda^3)e^{\mu_1(\lambda)x} \rho(\text{Re}(\mu_1(\lambda)x))-r_1(i\lambda^3)e^{\mu_1(\lambda)x}\rho(\text{Re}(\mu_2(\lambda)x))}{r_2(i\lambda^3)-r_1(i\lambda^3)}\lambda^2\widetilde{h}(\lambda^3)\right),  \\
    &\mathcal{F}\{\mathfrak{g}\}={3}\left(\frac{e^{\mu_2(\lambda)x}\rho(\mu_2(\lambda)x) -e^{\mu_1(\lambda)x}\rho(\mu_1(\lambda)x)}{r_2(i\lambda^3)-r_1(i\lambda^3)} \right)\lambda^2\widetilde{g}(\lambda^3)),
\end{align}
where $\rho(x)$ is a smooth function in $\R$ supported on $(-1, \infty)$, and $\rho(x)=1$ for $x\geq 0$.

\begin{remark}
The boundary operator \eqref{Wb} can be represented as well as inverse Fourier transform in time, this representation will help us to conclude about regularity in time.
    \begin{align}
    &W_{D}(t)h=\mathcal F_{\tau\to t}^{-1}\left\{\frac{r_2(i\tau)e^{r_1(i\tau)x}-r_1(i\tau)e^{r_2(i\tau)x}}{r_2(i\tau)-r_1(i\tau)}\widetilde{h}(i\tau) \right\}, \label{123}\\
    &W_{N}(t)h=\mathcal F_{\tau\to t}^{-1}\left\{\frac{e^{r_2(i\tau)x} -e^{r_1(i\tau)x}}{r_2(i\tau)-r_1(i\tau)}\widetilde{g}(i\tau)\right\}.\label{124}
\end{align}
\end{remark}
 
The following lemma deals with the behavior of the spatial traces of the boundary operators.

\begin{lemma}
 Let $\alpha \in \mathbb R$ be given, for $h$ and $g$ satisfying $\chi_{(0, \infty)} h \in \h^{\alpha}(\mathbb{R}) \quad \text{and} \quad \chi_{(0, \infty)} g\in \h^{\alpha-\frac{1}{3}}(\mathbb{R})$ then 
   % $W_{D}(t)h$, $W_{D}(t)g \in \mathbf L_x^\infty(\mathbb R; \h_t^{\alpha}(\mathbb R))$ and
$$
\sup _{x \in \mathbb R^+}\left\|W_{D}(t) h\right\|_{\h_t^{\alpha}(\mathbb R)} \lesssim_s \|\chi_{(0,\infty)}h\|_{\h^\alpha(\mathbb R)}, \quad
\sup _{x \in \mathbb R^+}\left\| W_{N}(t) g\right\|_{\h_t^{\alpha}(\mathbb R)} \lesssim \|\chi_{(0,\infty)}g\|_{\h^{\alpha-\frac{1}{3}}(\mathbb R)}.
$$
\end{lemma}
\begin{proof}
From the characterization of the boundary operators  \eqref{123} and \eqref{124}, together with    
    $$\frac{r_2(i\tau)e^{r_1(i\tau)x}-r_1(i\tau)e^{r_2(i\tau)x}}{r_2(i\tau)-r_1(i\tau)}\sim O(1) \quad \text{ and }\quad \frac{e^{r_2(i\tau)x} -e^{r_1(i\tau)x}}{r_2(i\tau)-r_1(i\tau)}\sim O(\langle \tau \rangle^{-\frac{1}{3}}),$$
the following estimates are obtained
    \begin{align*}
        \left\| W_{D}(t)h\right\|_{\h_t^{\alpha}(\mathbb{R})}^2 &\lesssim \left\|\langle \tau \rangle^{\alpha}\frac{r_2(i\tau)e^{r_1(i\tau)|x|}-r_1(i\tau)e^{r_2(i\tau)|x|}}{r_2(i\tau)-r_1(i\tau)}\widetilde{h}(i\tau)\right\|_{ L^2(\mathbb{R})}^2\\
        &=\int_{\mathbb R}\langle \tau \rangle^{2\alpha}\left|\frac{r_2(i\tau)e^{r_1(i\tau)|x|}-r_1(i\tau)e^{r_2(i\tau)|x|}}{r_2(i\tau)-r_1(i\tau)}\right|^2|\widetilde{h}(i\tau)|^2d\tau\\
        &\lesssim \int_{\mathbb R}\langle \tau \rangle^{2\alpha}|\widetilde{h}(i\tau)|^2d\tau=\|\chi_{(0,\infty)}h\|_{\h^\alpha(\mathbb R)}.
    \end{align*}
Similarly 
  \begin{align*}
        \left\| W_{N}(t)g\right\|_{\h_t^{\alpha}(\mathbb{R})}^2 &\lesssim \left\|\frac{e^{r_2(i\tau)x} -e^{r_1(i\tau)|x|}}{r_2(i\tau)-r_1(i\tau)}\widetilde{g}(i\tau)\right\|_{ L^2(\mathbb{R})}^2=\int_{\mathbb R}\langle \tau \rangle^{2\alpha}\left|\frac{e^{r_2(i\tau)|x|} -e^{r_1(i\tau)x}}{r_2(i\tau)-r_1(i\tau)}\right|^2|\widetilde{g}(i\tau)|^2d\tau\\
        &\lesssim \int_{\mathbb R}\langle \tau \rangle^{2\alpha-\frac{2}{3}}|\widetilde{g}(i\tau)|^2d\tau=\|\chi_{(0,\infty)}g\|_{\h^{\alpha-\frac{1}{3}}(\mathbb R)}.
    \end{align*}
    \end{proof}

For $\alpha\geq 0$, we  have the embedding $H^{\frac{\alpha-1}{3}}(\mathbb R)\subset H^{\alpha-\frac{1}{3}}(\R)$, therefore
\begin{equation}
 \left\| W_{N}(t)g\right\|_{\h_t^{\alpha}(\mathbb{R})}\lesssim  \|\chi_{(0,\infty)}g\|_{\h^{\frac{\alpha-1}{3}}(\mathbb R)} .
\end{equation}

Now, our attention is focused on the space regularity of Boundary operators.
\begin{proposition}\label{boudnaryoperators1}
       Let $s \geq 0$, $h$ and $g$ satisfying $\chi_{(0, \infty)} h \in \h^{\frac{s+1}{3}}(\mathbb{R})$, $\chi_{(0, \infty)} g\in \h^{\frac{s}{3}}(\mathbb{R})$, then 
    $W_{D}(t)h, W_{D}(t)g \in  L^\infty(\mathbb R; \h^s(\mathbb R))$ and
    \begin{equation}\label{Halpha dirichlet}
        \sup _{t \in \mathbb R}\left\|W_{D}(t)h\right\|_{\h_x^s(\mathbb R)} \lesssim \|h\|_{\h^{\frac{s+1}{3}}\left(\mathbb R^{+}\right)},
    \end{equation}
\begin{equation}\label{Wng}
\sup _{t \in \mathbb R}\left\|W_{N}(t)g\right\|_{\h_x^s(\mathbb R)} \lesssim \|g\|_{\h^{\frac{s}{3}}{\left(\mathbb R^{+}\right)}} .
\end{equation}
\end{proposition}
\begin{proof}
From  the representation of the 
boundary operators in \eqref{12345}, for every $t>0$, we have
\begin{align}  
    \partial_x^n W_D(t)h=\frac{3}{2\pi i}\int_{\mathbb R}  e^{i\lambda x}\left(r_1^n(i\lambda^3)p_1(x,\lambda)-r_2^n(i\lambda^3)p_2(x,\lambda)\right)\lambda^{2}\widetilde{h}(\lambda^3)d\lambda,\label{derivada W}
\end{align}
where $$p_1(x,\lambda):=e^{i\lambda^3 t}\frac{r_2(i\lambda^3)e^{\mu_1(\lambda)x}\rho(\text{Re}(\mu_1(\lambda)x))}{r_2(i\lambda^3)-r_1(i\lambda^3)},$$ 
$$p_2(x,\lambda):=e^{i\lambda^3 t}\frac{r_1(i\lambda)e^{\mu_2(\lambda^3)x}\rho(\text{Re}(\mu_2(\lambda)x))}{r_2(i\lambda^3)-r_1(i\lambda^3)},$$
with $\rho(x)$ a smooth function supported on $(-\infty,1)$, and $\rho(x)=1$ for $x\leq 0$ and 
$n\in \mathbb N$. As you can see, the symbols depend on the variable $t$. We have omitted this dependence because the bound is uniform in time.\\

$\partial_x^n W_D(t)h$ can be written in terms of two pseudo-differential operators with symbols $p_1$ and $p_2$. Indeed,  for 
$$ \Psi_p f(x)=\frac{1}{2\pi i}\int_{\mathbb{R}} e^{ i x \cdot \xi} p(x, \xi) \widehat{f}(\xi) d \xi, \quad f \in \mathcal{S}\left(\mathbb{R}^n\right), $$
see Appendix \ref{useful results} for more details, we have
\begin{equation}
     [\partial_x^n W_D(t)h](x)=3\Psi_{p_1}(\nu_1)(x)+3\Psi_{p_2}(\nu_2)(x),
\end{equation} 
where $\nu_j$ is given via its Fourier transform as $\widehat{\nu_j}(\lambda)=r_j^n(i\lambda^3)\lambda^{2}\widetilde{h}(\lambda^3).$
It is important to observe $p_j\in S^0$, $j=1,2$ where the family of symbols $S^m$ is defined in Appendix \eqref{simbol}. Thus, by Theorem \ref{Sobolev boundedness}, we have
$$\|\partial_x^n W_D(t)h\|_{L^2(\mathbb R)}\lesssim \| \nu_1\|_{\mathbf L^2(\mathbb R)}+\| \nu_2\|_{\mathbf L^2(\mathbb R)} \lesssim \|\chi_{(0, \infty)} h\|_{\h^{\frac{n+1}{3}}(\mathbb R)}.$$
By interpolation, inequality \eqref{Halpha dirichlet} follows. Similarly, the inequality \eqref{Wng} can be obtained.
\end{proof}
The following proposition estimates the norms of the Dirichlet and Neumann boundary operators in Bourgain space. The proof presented here draws inspiration from the ideas developed in \cite{erdougan2016regularity}.

\begin{proposition}\label{boundary operator}
    Let $s\geq 0$ and $0<b<\frac{1}{18}$. For $\varphi \in C^{\infty}_0(\R)$ (smooth function with compact support), then there exists a constant $C:=C(s,b,\varphi)>0$ such that
    \begin{gather}\label{3.22}
   \left\|\varphi(t) W_D(t) h\right\|_{X^{s, b}} \lesssim\left\|\chi_{(0, \infty)} h\right\|_{H^{\frac{s+1}{3}}(\mathbb{R})}, 
    \end{gather}
    and
     \begin{gather}\label{3.23}
   \left\|\varphi(t) W_N(t) g\right\|_{X^{s, b}} \lesssim\left\|\chi_{(0, \infty)} g\right\|_{H^{\frac{s}{3}}(\mathbb{R})} .
    \end{gather}
\end{proposition}
\begin{proof}
By interpolation arguments, we can reduce the problem to the case where $n$ is a natural number and $b$ lies between 0 and 1/2. Now, for any smooth function $\varphi$ with compact support, taking the Fourier transform of both sides of equation \eqref{derivada W} with respect to both time and space, we obtain
\begin{multline*}  
    \mathcal F_{tx}\left\{\varphi(t)\partial_x^nW_{D}(t)h \right\}(\tau,\xi)= \mathcal{F}_t\{\varphi(t)\}*\mathcal{F}_t\mathcal{F}_x\{\partial_x^nW_{D}(t)h\}\\
    =\int_{\mathbb R}\widehat{\varphi}(\tau-\lambda^3)\frac{r_2(i\lambda^3)r^n_1(i\lambda^3)\widehat{\upsilon_1}(\xi,\lambda)-r_1(i\lambda^3)r^n_2(i\lambda^3)\widehat{\upsilon_2}(\xi,\lambda)}{r_2(i\lambda^3)-r_1(i\lambda^3)}\lambda^2\widetilde{h}(i\lambda^3)d\lambda,
\end{multline*} where 
$$\upsilon_j(x,\lambda)=e^{ix(\lambda+\text{Im}(\mu_j(\lambda)))}\varsigma(\text{Re}(\mu_j(\lambda)x)), $$ with $\varsigma(x)=e^{x} \rho(x)$, where $\rho(x)$ is a smooth function supported on $(-\infty,1)$, and $\rho(x)=1$ for $x\leq 0$. Note that $\varsigma\in S(\mathbb R)$, this implies
$$|\widehat{\upsilon_j}(\xi,\lambda)|\lesssim  \frac{|\text{Re}(\mu_j(\lambda))|^m}{|\text{Re}(\mu_j(\lambda))|^m+|\xi-(\lambda +\text{Im}(\mu_j(\lambda)))|^m}, \quad m\in \mathbb N.$$
Thus
$$|\mathcal F_{tx}\left\{\varphi(t)\partial_x^nW_{D}(t)h \right\}(\xi,\tau)|\lesssim \int_{\mathbb R}|\widehat{\varphi}(\tau-\lambda^3)|\left( |\widehat{\upsilon_1}(\xi,\lambda)|+|\widehat{\upsilon_2}(\xi,\lambda)|\right)|\lambda|^{2+n}|\widetilde{h}(i\lambda^3)|d\lambda,$$ therefore
\begin{align}
   \left\|\varphi(t)\partial_x^nW_{D}(t)h \right\|_{X^{n, b}}& \lesssim\left\|\left\langle i(\tau+\xi^3) +\xi^2\right\rangle^{b} \int_{\mathbb R}|\widehat{\varphi}(\tau-\lambda^3)|\left( |\widehat{\upsilon_1}(\xi,\lambda)|+|\widehat{\upsilon_2}(\xi,\lambda)|\right)|\lambda|^{2+n}|\widetilde{h}(i\lambda^3)|d\lambda \right\|_{L_{\xi}^2 L_\tau^2} \notag \\
   & \lesssim  \left\|\left\langle i(\tau+\xi^3) +\xi^2\right\rangle^{b} \int_{\mathbb R}\frac{|\widehat{\varphi}(\tau-\lambda^3)| |\lambda|^{2+n+
m}}{|\lambda|^m+|\xi|^m} |\widetilde{h}(i\lambda^3)|d\lambda \right\|_{L_{\xi}^2 L_\tau^2}.
\end{align}
We partition the domain of integration into two disjoint subsets:  for any $m\in \N$, let $D_1^{{m}} = \{(\xi,\lambda) : |\lambda|^m + |\xi|^m > 1\}$ and $D_2^{{m}} = \{(\xi,\lambda) : |\lambda|^m + |\xi|^m \leq 1\}$.  Let $I_1$ and $I_2$ denote the integrals of $D_1^{{m}}$ and $D_2^{{m}}$, respectively. To estimate $I_1$, we observe that
 $\left|\widehat{\varphi}\left(\tau-\lambda^3\right)\right| \lesssim\left\langle\tau-\lambda^3\right\rangle^{-2},$ $ \left\langle i(\tau+\xi^3) +\xi^2\right\rangle \lesssim \langle \xi\rangle^2\left\langle\tau-\lambda^3\right\rangle\left\langle\lambda+\xi\right\rangle^3\lesssim \left\langle\tau-\lambda^3\right\rangle\left\langle|\lambda|+|\xi|\right\rangle^3$, and $|\lambda|^m+|\xi|^m \sim\left\langle |\lambda|^m+|\xi|^m\right\rangle$. Thus we have the bound
  \begin{align*}
     I_1&\lesssim \left\| \int_{\mathbb R}\frac{|\lambda|^{2+n+
m}}{\langle \tau-\lambda^3\rangle^{2-b}(|\lambda|+|\xi|)^{m-3b}}|\widetilde{h}(i\lambda^3)|d\lambda \right\|_{L_{\xi}^2 L_\tau^2}\\
&\lesssim  \left\| \int_{\mathbb R}\frac{|\lambda|^{2+n+
m}}{\langle \tau-\lambda^3\rangle^{2-b}}\left\|\frac{1}{(|\lambda|+|\xi|)^{m-3b}}\right\|_{L_\xi^2}|\widetilde{h}(i\lambda^3)|d\lambda \right\|_{ L_\tau^2}.
 \end{align*}
 Note that 
 $$\left\|\frac{1}{(|\lambda|+|\xi|)^{m-3b}}\right\|_{L_\xi^2}\lesssim |\lambda|^{-(m-3b-\frac{1}{2})},$$ since $b\in (0,\frac{1}{18})$ this implies 
 \begin{align*}
     I_1& \lesssim  \left\| \int_{\mathbb R}\frac{|\lambda|^{\frac{5}{2}+n+3b}}{\langle \tau-\lambda^3\rangle^{2-b}}|\widetilde{h}(i\lambda^3)|d\lambda \right\|_{ L_\tau^2} \lesssim  \left\| \int_{\mathbb R}\frac{|w|^{\frac{1+n}{3}}}{\langle \tau-w\rangle^{\frac{3}{2}}}|\widetilde{h}(iw)|dw \right\|_{ L_\tau^2},
 \end{align*}
 using  Young’s inequality, we have  
 $$I_1\lesssim \left\|\langle \tau \rangle^{-\frac{3}{2}}\right\|_{L_{{\tau}}^1}\left\||\tau|^{\frac{1+n}{3}}\widehat{h}(\tau)\right\|_{L^2_{\tau}} \lesssim\left\|\chi_{(0, \infty)} h\right\|_{ H^{\frac{1+n}{3}}(\mathbb{R})}.$$
In the latter case, using Minkowski’s inequality we get
\begin{align*}
    I_2&\lesssim \left\|\langle\tau\rangle^{\frac{1}{2}} \int_0^1\langle\tau\rangle^{-2} \frac{\lambda^{2+n+m}}{|\lambda|^m+|\xi|^m}\left|\widehat{h}\left(i \lambda^3\right)\right| d \lambda\right\|_{ L_{|\xi| \leq 1}^2   L_\tau^2}
    \lesssim  \int_0^1\left\| \frac{\lambda^{m}}{|\lambda|^m+|\xi|^m}\right\|_{ L_{|\xi| \leq 1}^2 }\lambda^{2+n} \left|\widehat{h}\left(i \lambda^3\right)\right| d \lambda\\
    &\lesssim \int_0^1
    |\lambda|^{1+n} \left|\widehat{h}\left(i \lambda^3\right)\right| d \lambda\lesssim \left\|\chi_{(0, \infty)} h\right\|_{ H^{\frac{1+n}{3}}(\mathbb{R})}.
\end{align*}
This completes the proof of inequality \eqref{3.22}. Inequality \eqref{3.23} can be obtained similarly; for this reason, we omit the details.
\end{proof}

\subsubsection{\bf Initial value problem in $\R^+$}
	For the linear problem with zero boundary conditions and non-homogeneous initial data.
		\begin{equation} \label{semi}                   \begin{cases}u_{t}-u_{xxx}-u_{xx}=0,\ &t>0, \ x>0,\\
			u(x,0)=u_0, \ &x>0,\\
			u(0,t)=0, \ u_x(0,t)=0, \ &t > 0, 
		\end{cases}
	\end{equation}
	
the solution of \eqref{semi} will be denoted by 
\begin{gather}\label{semisol}
 u(x,t)= W_{0}(t) u_0(x),
\end{gather}
	
with  $W_0$ the  $\mathbb{C}_0$-semigroup in the space $L^2(\R^+)$ generated by the operator $Av=v''' + v''$ with domain $\mathcal{D}(A):= \{ v\in H^3(\R^+) : v(0)=0, v'(0)=0\}$.\\

 If we consider $\Bar{u}$ the solution of Real extension of the initial value problem \eqref{semi}  by 
	\begin{equation} \label{semiExtention} 
	\begin{cases}
	\Bar{u}_{t}-\Bar{u}_{xxx}-\Bar{u}_{xx}=0,\ &t>0, \ x\in \R,\\
	\Bar{u}(x,0)=\Bar{u}_0(x),& x\in \R,
    \end{cases}
	\end{equation}
this solution will be given by $\Bar{u}(x,t)= W_{\R}(t)\Bar{u}_0(x)$. And from the operator $W_{\R}$ in \eqref{WR} and the boundary operator $W_{b}$ in \eqref{Wb} the solution of \eqref{lineal} can be written as 
		$$u(x,t)= W_{0}(t)u_0(x)= W_{\R}(t)\Bar{u}_0(x)- W_{bdr}(t)(\tilde{h},\tilde{g}), $$
     with    $\tilde{h}(t)=\Bar{u}(0,t)$ and $\tilde{g}(t)=\Bar{u}_x(0,t)$ are the traces of the solution on the hole line.

\section{Well-posedness of the non-linear KdV-B}
In this section, we study the well-posedness of the IBVP \eqref{1.1}. A proof of Theorem \ref{wellposedness} is provided as follows:
\begin{proof}%{\bf of Theorem \ref{wellposedness}}\\
  Let $u_0 \in H^s\left(\mathbb R^{+}\right)$, $h \in H_{\text {Loc }}^{\frac{1+s}{3}}\left(\mathbb R^{+}\right)$ and $g\in H_{\text {Loc }}^{\frac{s}{3}}\left(\mathbb R^{+}\right) $ be given with $s \in[0,2-3b)$.  
   We denote by  $W_{\mathbb R}$  the linear operator on the hole line for the problem \eqref{1.2}, by $u_0^*$ a real extension of $u_0$  such that $\|u_0^*\|_{H^s(\mathbb R)}\leq \|u_0\|_{ H^s (\mathbb R^+)}$. \\
For $T>0,$ $b\in (0,\frac{1}{18})$ and  $\varphi \in C^{\infty}_0(\mathbb R)$, let us define the operator $$\mathbf{\Gamma}: \mathcal{V}_{s, b}(\R^+\times (0,T)) \to \mathcal{V}_{s, b}(\R^+\times (0,T)),$$  
    such that 
\begin{multline}\label{nlsol}
\mathbf{\Gamma}(w):=\varphi(t)W_{\mathbb R}(t) {u}^*_0(x)\\+\varphi(t) \mathcal{W}_{b d r}(t) [h-p_1-q_1,g-p_2-q_2]+\varphi(t)\int_0^t \mathcal{W}_\mathbb R(t-\tau)f(w)(x,\tau) d \tau
\end{multline}
with $f(w)=ww_x$,
\begin{align*} 
    p_1(t)&=D_0\left(W_{\mathbb{R}} u^*_0\right), 
    &   p_2(t)&=D_0\left(\partial_x W_{\mathbb{R}} u^*_0\right),\\
    q_1(t)&=D_0\left(\int_0^t W_{\mathbb{R}}\left(t-t^{\prime}\right) f(w) d t^{\prime}\right),
    & q_2(t)&=D_0\left(\int_0^t \partial_x W_{\mathbb{R}}\left(t-t^{\prime}\right) f(w) d t^{\prime}\right) .
\end{align*}
By Proposition \ref{RlinearID}, 
$$
\left\|\varphi(t)W_{\mathbb R}(t) u_0 \right\|_{\mathcal{V}_{s, b}} \lesssim\left\|{u}^*_0\right\|_{H^s(\mathbb R)} \lesssim\|u_0\|_{H^s\left(\mathbb{R}^{+}\right)},
$$
from Propositions \ref{RlinearForcing},\, \ref{Bourgain estimate prop}  and \ref{bilinearestimate}, there exists $0<\mu<1$ such that
$$\left\|\varphi(t) \int_0^t W_{\mathbb{R}}\left(t-t^{\prime}\right) f(u) d t^{\prime}\right\|_{\mathcal{V}_{s, b}} \lesssim\|f(u)\|_{\mathcal{V}_{s,-b}} \lesssim T^{\mu }\left\|u\right\|^2_{\mathcal{V}_{s,b}},$$
and regarding the boundary operator, 
$$\mathcal{W}_{b d r}(t) [h-p_1-q_1,g-p_2-q_2]= W_D(t)[h-p_1-q_1]+W_N(t)[g-p_2-q_2].$$
Additionally,  Propositions \eqref{RlinearIDRegularity}, \eqref{boudnaryoperators1} and \eqref{boundary operator}, imply
$$\|\varphi(t)W_D(t)[h-p_1-q_1]\|_{\mathcal{V}_{s,b}}\lesssim \|\chi_{(0,\infty)}(h-p_1-q_1)\|_{\h^{\frac{s+1}{3}}(\mathbb R)} \lesssim \left(\|u_0\|_{H^s\left(\mathbb R^{+}\right)}+\|h\|_{H^{\frac{s+1}{3}}(0,T)}\right)$$
and
$$\|\varphi(t)W_N(t)[g-p_2-q_2]\|_{\mathcal{V}^{s,b}}\lesssim \|\chi_{(0,\infty)}(g-p_2-q_2)\|_{\h^{\frac{s}{3}}(\mathbb R)} \lesssim \left(\|u_0\|_{H^s\left(\mathbb R^{+}\right)}+\|g\|_{H^{\frac{s}{3}}(0,T)}\right).$$

The operator $\mathbf{\Gamma}:B_r\to B_r$ is a contraction for  a suitable $r$, in effect, for $w \in B_{r}$ combining the estimates and selecting 
$$\|\mathbf{\Gamma} w\|_{\mathcal{V}^{s,b}} \lesssim \|u_0\|_{H^s\left(\mathbb{R}^{+}\right)}+\|h\|_{H^{\frac{ s+1}{3}}(0,T)}+ \|g\|_{H^{\frac{s}{3}}(0, T)}+T^{\mu}\|w\|_{\mathcal{V}^{s,b}}^2 \le r \delta+ T^{\mu} r^2< r (\delta +r T^{\mu})<r, $$
and for $w, v \in B_r$,
$$
\|\mathbf{\Gamma}(v) -\mathbf{\Gamma} (w)\|_{\mathcal{V}^{s,b}} \lesssim T^\mu\|v-w\|_{\mathcal{V}^{s,b}}\|v+w\|_{\mathcal{V}^{s,b}}.
$$
and, so
$$
\|\mathbf{\Gamma}(w)-\mathbf{\Gamma}(v)\|_{\mathcal{V}_{s,b}}\left(\mathbb R^{+} \times(0, T)\right) \leqslant 2 \delta r T^{\mu} \|w-v\|_{\chi^{s, \frac{1}{2}}\left(\mathbb R^{+} \times(0, T)\right)}.
$$

selecting $\delta>0$ and $0<\mu< 1$ such that $ \delta + r T^{\mu}<1$  and  $ 2 \delta r T^{\mu}<1$ the function $\mathbf{\Gamma}$ is a contraction mapping of the ball $B_r$. This implies the existence of a fixed point $u$ in $\mathcal{V}_{s,b}$, which is the desired solution.

Next, suppose that $u_0 \in H^3\left(\mathbb R^{+}\right)$, $h \in H_{\text {Loc }}^{\frac{4}{3}}\left(\mathbb R^{+}\right)$ and $g\in H_{\text {Loc }}^{1}\left(\mathbb R^{+}\right)$ such that 
$$ \|u_0\|_{H^s\left(\mathbb{R}^{+}\right)}+\|h\|_{H^{\frac{ s+1}{3}}(0,T)}+ \|g\|_{H^{\frac{s}{3}}(0, T)} \leq r.$$ By applying the previously established result, we know that there exist  a  solution$u\in \mathcal{V}_{1, b}(\R^+\times (0,T))$  to the problem \eqref{1.2}. Let $v = u_t$. Then $v$ solves also the IBVP \eqref{1.2} but with the following initial and boundary conditions:
$$v(x,0)=u_t(x,0)\in H^0(\mathbb R^+), \ \ v(x,0)=h'(t)\in H^{\frac{1}{3}}(\mathbb R^+), \ v_t(0,t)=g'(t)\in L^2(\mathbb R^+),$$ this problem has a unique solution $v\in \mathcal{V}_{0, b}(\R^+\times (0,T)) $, thus we have $v=u_t\in \mathcal{V}_{0, b}(\R^+\times (0,T)) $ therefore $u \in  \mathcal{V}_{3, b}(\R^+\times (0,T)) $.

By non-linear interpolation theory, we conclude that for  $u_0 \in H^s\left(\mathbb R^{+}\right)$, $h \in H_{\text {Loc }}^{\frac{s+1}{3}}\left(\mathbb R^{+}\right)$ and $g\in H_{\text {Loc }}^{\frac{s}{3}}\left(\mathbb R^{+}\right)$, where  $s\in(0,3)$,   the corresponding solution $u\in C((0,T); H^s(\mathbb R^+))$. The proof of the remaining case for s is similar and therefore omitted.
\end{proof}

\section{Carleman's Estimate}

A global Carleman's estimate is computed for the KdV-Burger equation on $\SLT:= [-L,L]\times [0,T] $ for $T>0$ and $L>0$. Let us define the set 

\begin{gather}\label{Uset}{\Upsilon}:= \Big\{ q \in C^3([0,T]);\ q(\pm L,t)=q_x(\pm L,t)=q_{xx}(\pm L,t)=0, \quad \text{for} \ 0\le t \le T\Big\}.\end{gather}

The following  Carleman's estimate is true for the KdV-Burger equation  

\begin{proposition} There exist a smooth  positive function $\psi$ defined on $[-L,L]\times[0,T]$, a constant  $s_0=s_0(L,T)$ and $C=C(L,T)$ such that for any $s\ge s_0$ and any $q\in \Upsilon$
\begin{multline}\label{Carleman}
 \int\limits_{\SLT} \left( \displaystyle\frac{s^5  }{t^5(T-t)^5} q^2+ \displaystyle\frac{ s^3   }{t^3 (T-t)^3} (q_x)^2 +  \frac{ s}{t (T-t)} (q_{xx})^2 \right)e^{-2s \psi} dxdt \le \\ \int\limits_{\SLT}  C ( q_t- q_{xx}-q_{xxx})^2 e^{-2s \psi}dxdt 
\end{multline}
\end{proposition}
\begin{proof}
For $s>0$,  $\phi=\phi(x)\in \mathcal{C}^3([-L,L])$, $q\in \Upsilon $ and  $\psi(x,t):= \displaystyle\frac{\phi(x)}{t(T-t)}$.  Consider 
    \begin{gather*}
    u:=  e^{-s \psi}q \quad \text{and} \quad  w:= e^{-s \psi} P(e^{s\psi}u)
    \end{gather*}
when $P=\partial_t - \partial_{xxx} - \partial_{xx}$ is the linear  operator. By direct computations, we obtain that
\begin{multline*}
    P[e^{s \psi}u]= e^{s \psi} \left\{ (u_t- u_{xx}-u_{xxx}) \right.\\+  \left(s \left[\psi_t- \psi_{xx}- \psi_{xxx}\right] -s^2\left[(\psi_x)^2 +3 \psi_x \psi_{xx}\right]-s^3 \left[(\psi_x)^3 + 3 \psi_x \psi_{xx}\right]\right)u  \\ +\left[-s\left(2\psi_x+ 3\psi_{xx}\right) -3s^2(\psi_x)^2\right]u_x+ 
    - 3s \psi_{x}  u_{xx}
    \left. \right\}
\end{multline*}
and therefore  \begin{gather}\label{w}
w= Au +Bu_x + Cu_{xx}-u_{xxx}+u_t
\end{gather} considering  \begin{eqnarray*}
    A&=& s \left[\psi_t- \psi_{xx}- \psi_{xxx}\right] -s^2\left[(\psi_x)^2 +3 \psi_x \psi_{xx}\right]-s^3 \left[(\psi_x)^3 + 3 \psi_x \psi_{xx}\right], \\
    B&=& -s\left(2\psi_x+ 3\psi_{xx}\right) -3s^2(\psi_x)^2, \\
    C&=&  -1 - 3s \psi_{x}. 
\end{eqnarray*}
Let us call  \begin{gather}\label{Ms}
    M_1(u):= u_t - u_{xxx} + Bu_x \quad \text{and} \quad
    M_2(u):= Au + C u_{xx},
\end{gather}
it is clear that 
\begin{equation}\label{inew}
2 \int\limits_{\SLT} M_1(u)M_2(u)dxdt\le \int\limits_{\SLT} \Big(M_1(u)+M_2(u)\Big)^2dxdt = \int\limits_{\SLT} w^2 dxdt,
\end{equation}
the left-hand side of \eqref{inew} can be compute as 
$\int\limits_{\SLT} M_1(u)M_2(u)dxdt = I + II$
with 
\begin{eqnarray*}
    I&=& \int\limits_{\SLT} M_1(u) Au dxdt=\int\limits_{\SLT} (u_t - u_{xxx} + Bu_x) Au dxdt\\
   % &=& \int\limits_{\SLT} Auu_t - Auu_{xxx} + AuBu_x  dxdt\\
    &=& \frac12\int\limits_{\SLT} \left((-A_t  +  A_{xxx} - (AB)_x)u^2 \right) dxdt - \frac32\int\limits_{\SLT} A_x(u_x)^2 dxdt,\\
    II&=& \int\limits_{\SLT} M_1(u)C u_{xx}    dxdt\\
    &=&  \int\limits_{\SLT} \left( u_t - u_{xxx} + Bu_x \right) Cu_{xx}dxdt= II_1 + II_2+II_3
\end{eqnarray*}
with 
\begin{eqnarray*}
    II_1&=& 
    \int\limits_{\SLT} C u_t u_{xx} dxdt= -\int\limits_{\SLT} (C u_t)_x u_{x} dxdt \\
    &=& - \int\limits_{\SLT}  \frac12 C \frac{d}{dt}(u_x)^2dxdt - \int\limits_{\SLT}  C_x (w-Au -Bu_x - Cu_{xx}+u_{xxx} ) u_x dt dx\\ 
     &=&   \int\limits_{\SLT}  \left(\frac12 C_t (u_x)^2 -  C_xwu_x -\frac12 (C_xA)_x (u^2)
     +C_xB(u_x)^2 -\frac12 (C_xC)_x(u_{x})^2-\frac12 C_{xxx}(u_x)^2 \right. \\ && \hspace{2cm} +C_x(u_{xx})^2\Big) dxdt \\
      &=&   \int\limits_{\SLT}  \left(   -\frac12 (C_xA)_x u^2+ 
      \left(\frac12 C_t  
     +C_xB -\frac12 (C_xC)_x-\frac12 C_{xxx}\right)(u_x)^2 +C_x(u_{xx})^2 -  C_xwu_x \right) dxdt, \\
     II_2&=&  \int\limits_{\SLT} \frac{1}{2} C_x(u_{xx})^2 dxdt,\\
     II_3&=&   \int\limits_{\SLT} BCu_xu_{xx}dxdt= - \int\limits_{\SLT} \frac12 (BC)_x (u_x)^2 dxdt.
\end{eqnarray*}

Therefore,  putting everything together 

\begin{multline}
   2  \int\limits_{\SLT}  M_1(u)M_2(u)dx dt  =\\
   \int\limits_{\SLT} \left(- A_t  +  A_{xxx} - (AB)_x   - (C_xA)_x   \right)u^2 +\\
      \left(   
      \left( C_t  
     +2C_xB - (C_xC)_x- C_{xxx} -  (BC)_x\right)(u_x)^2 + 3 C_x  (u_{xx})^2 -  2 C_xwu_x \right) dxdt \\ \le  \int\limits_{\SLT} w^2 dxdt
\end{multline}
and  putting together the terms of $w$,  $$ \int\limits_{\SLT} \left(2c_xw u_x + w^2\right)dxdt \le  \epsilon\int\limits_{\SLT}    C_x^2(u_x)^2 dxdt + \left(\frac{1}{\epsilon}+ 1 \right) \int\limits_{\SLT} w^2 dxdt.  $$

If we set 
\begin{eqnarray}\label{DEF}
    D&=& -A_t  + A_{xxx} - (AB)_x   - (C_xA)_x ,  \notag\\
    E&=&   C_t  
     +2C_xB - (C_xC)_x- C_{xxx} -  (BC)_x - \epsilon C_x^2, \notag\\
    F&=&   3 C_x,
\end{eqnarray}
then 
 \begin{equation}\label{carleman1}
 \int\limits_{\SLT} \left(D u^2
      +   
      E(u_x)^2 +   F(u_{xx})^2 \right) dxdt \le   \left(\frac{1}{\epsilon}+ 1 \right)\int\int\limits_{\SLT} w^2 dxdt.
\end{equation}

By direct computations, 
\begin{eqnarray*}
    D &=& - 15 \phi'(x)^4 \phi''(x) \frac{s^5}{t^5(T-t)^5}+ \displaystyle\frac{1}{t^4(T-t)^4}  \mathcal{O}(s^4),\\
     E&=&-9\phi '(x)^2 \phi ''(x)\frac{ s^3 }{t^3 (T-t)^3}+ \frac{1}{t^4(T-t)^4}  \mathcal{O}(s^2),\\
     F&=&-9  \frac{ \phi ''(x)}{t (T-t)},
\end{eqnarray*}
 there exist $C_D,C_E$ and $C_F$ positive constants such that
 if  $\phi''<0$ for $x\in [-L,L]$, we obtain 
 \begin{eqnarray*}
 D &\ge&   - 15 \phi'(x)^4 \phi''(x) \displaystyle\frac{s^5 C_D }{t^5(T-t)^5},\\
E&\ge&  -9\phi '(x)^2 \phi ''(x) \displaystyle\frac{ s^3 C_E  }{t^3 (T-t)^3},\\
 F &\ge& C_F  \frac{ s}{t (T-t)},
\end{eqnarray*}
 for $s$ large enough, when it is imposed that   $\phi''<0$ for $x\in [-L,L]$, the expressions $D,E$ and $F$, in \eqref{DEF}, will be positives $x\in [-L,L]$ and $t\in [0,T]$. 
 Therefore, for $x\in[-L,L]$ we select  the function $$\phi(x)= -x^2 - (2L+\frac32 T)x.$$

From \eqref{carleman1}, for $s$ large enough, we have for $C=C(L)$
\begin{equation}\label{carleman3}
\int\limits_{\SLT} \left( \displaystyle\frac{s^5  }{t^5(T-t)^5} u^2+ \displaystyle\frac{ s^3   }{t^3 (T-t)^3} (u_x)^2 +  \frac{ s}{t (T-t)} (u_{xx})^2 \right) dxdt \le   C \int\limits_{\SLT} w^2 dxdt.
\end{equation}
 Returning to the $q$ function,  we have that since $u=e^{s \psi}q$, $u_x=\left(-s \psi_x  q+ q_x\right)e^{-s \psi}$ and $u_{xx}= \left( \left(s^2 (\psi_x)^2-s \psi_{xx}\right)q - 2s\psi_x  q_x + q_{xx}\right)e^{-s \psi}$
for $s$ large enough we have that the higher terms  are 

\begin{eqnarray*}    
\int\limits_{\SLT} && ( q_t- q_{xx}-q_{xxx})^2 e^{-2s \psi}dxdt \ge \int\limits_{\SLT}  (Au +Bu_x + Cu_{xx}-u_{xxx}+u_t)^2 dxdt \ge\\ &&
\ge  C
\int\limits_{\SLT} \left( \displaystyle\frac{s^5  }{t^5(T-t)^5} q^2+ \displaystyle\frac{ s^3   }{t^3 (T-t)^3} (q_x)^2 +  \frac{ s}{t (T-t)} (q_{xx})^2 \right)e^{-2s \psi} dxdt
\end{eqnarray*}
with $C:=C(L,T).$
\end{proof}

\section{\bf Controllability}

In this section, we address the exact boundary controllability problem of the KdV-B equation.  Which is based on answering the question:\\

Is it possible to find  functions $h$ and $g$ in suitable spaces, such that for a giving time $T>0$, an initial state $u_0 \in L^2(\R)$ and a final state $u_T\in L^2(\R)$ the system 
\begin{equation}\label{dfcontrol}
    \begin{cases}
        u_t-u_{xxx}-u_{xx}=0, & 0\le t\le T, x>0\\
        u(0,t)=h(t), & 0\le t\le T\\
        u_x(0,t)=g(t), & 0\le t\le T\\
        u(x,0)=u_0
    \end{cases}
\end{equation}
has a solution such that $u(x,T)=u_T?$\\

The definition of approximate controllability is enough to ensure that the solution of \eqref{dfcontrol} satisfies that for some $\epsilon>0$, $\| u(x,T)-u_T\|< \epsilon$.  A general result for the approximate controllability  is stated in the following  proposition \cite[Proposition 1]{Rosierub}:

\begin{proposition}\label{aproxcontrol}
    Consider the operator 
    $$Au=\sum_{i=0}^n a_i \frac{d^i u}{dx^i},$$
    with $a_i\in\R$ for $i=0,1,\ldots,n$, $n\ge 2, a_n\neq 0$ and domain $\D(A):= \{ u\in L^2(\R): Au \in L^2(\R)\} $ such that $A$ generates a continuous semigroup $\{S(t)\}_{t\ge0}$ on $L^2(\R)$. Taking $T>0$ and $L_1 < L_2$ the set 
    $$\mathfrak{R}:=\left\{\int_0^T S(T-t)f(t,\cdot) dt : f\in L^2(\R^2), \text{supp} \, f \subset [L_1,L_2]\times [0,T] \right\}$$ is a strict dense subspace of $L^2(\R).$
\end{proposition}

Therefore, the approximated controllability of the control system \eqref{dfcontrol} is obtained by Proposition \eqref{aproxcontrol}. However, exactly controllability is more challenging because the compactness results are needed, and there is a lack of them in $\R^+$.   From now on we focus on the proof of the main results of controllability: Theorem \eqref{nonullcontrol} and \eqref{control}. 

\subsection{\bf Proof of first main result of controllability Theorem \ref{nonullcontrol}}
There is a lack of null controllability for the KdV-B equation for solutions with bounded energy, i.e. $u \in L^{\infty}(0,T;L^2(\Rm))$ independently of the boundary terms. In other words, the system \eqref{nonull} fails to be null 
controllable.\\

The proof argument by contradiction, creating a set of possible solutions that drives the control system \eqref{nonull} to zero, and by the construction of a linear homeomorphism, from a special quotient of $E$ space to initial data, and the help of lemmas \eqref{soln} and \eqref{lemma inequality} the contradiction can be formed.\\

\begin{proof} {\bf of Theorem \eqref{nonullcontrol}}
    Let us suppose that there exists  an initial data $u_0\in L^2(\R^+)$ such  that the solution $u\in L^{\infty} (0,T,L^2(\R^+))$ of \eqref{nonull} satisfies $u(x,T)=0$.\\

    Defining the space of solutions that are null at time $T$ by
    \begin{equation}\label{E}
      E:= \left\{ u \in  L^{\infty} (0,T,L^2(\R^+)) \quad \text{such that solves} \quad \eqref{nonull} \quad \text{and} \quad u(x,T)=0   
        \right\}
    \end{equation}
    with the norm of $ L^{\infty} (0,T,L^2(\R^+))$, it is a Banach space, i.e.  $\|u\|_E:= \|u\|_{L^{\infty} (0,T,L^2(\R^+))}$.   Moreover, 
    $$\|u(x,T)\|_{H^{-3}(\R^+)} \lesssim  \| u\|_{H^1(0,T,H^{-3}(\R^+))} \lesssim \|u\|_E,  \quad \forall u\in E.$$

    Consider the  mapping $$\Gamma: E \ni u \to u(x,0) \in H^{-3}(\R^+),$$ which is continuous taking values in $L^2(\R^+)$. Notice that it is also continuous from $E \to L^2(\R^+)$. \\

    Defining the quotient space $\tilde{E}:= E/{Ker \Gamma}$ with the norm $\|\Tilde{w}\|_{\Tilde{E}}:= \inf_{w \in \pi(E)}\|w\|_E$, let us defined the quotient function $$\widetilde{\Gamma}: \tilde{E} \ni \tilde{u} \to \tilde{u}(x,0) \in L^{2}(\R^+).$$
    For $u\in E$, we define the projection $\pi:E\to \Tilde{E}$  and therefore, whenever  \ $\tilde{\Gamma}(\pi(u))=\Gamma(u)$, $\Gamma$ has a continuous inverse by the open map theorem function. \\

    For any $\lambda>0$,  let $u^{\lambda}\in E$ such that $$\Gamma^{-1}(v^{\lambda}(x,0))=\pi(u^{\lambda}),$$
    and $$\|u^{\lambda}\|_E \le  2\|\pi(u^{\lambda})\|_{\tilde{E}}\le 2  \|\tilde{\Gamma}^{-1}\|\|v^{\lambda}(x,0)\|_{L^2(\R^+)}. $$

 For a  $L>0$,  for any $u^{\lambda}\in E$ and any $v^{\lambda}$  solution that satisfies  Lemma \eqref{soln}, we have
 \begin{gather*}
     \int_0^T \int_0^L(\partial_t-\partial_{xxx}-\partial_{xx})u^{\lambda} v^{\lambda} dxdt = 0,
 \end{gather*}
 therefore, by integration by part we obtain 
 \begin{gather}\label{ipp}
     - \int_0^L u^{\lambda}(x,0) v^{\lambda}(x,0)dx -\left(\langle u^{\lambda}_{xx}+ u^{\lambda}_x,v^{\lambda}\rangle - \langle u^{\lambda}_x-u^{\lambda},v^{\lambda}_x\rangle \Big |_0^L\right)=0.
 \end{gather}
 given that $\Gamma\left(\pi(u^{\lambda})\right)=v^{\lambda}(x,0),$ from \eqref{ipp}
 \begin{eqnarray}\label{auxa}
\int_0^L (v^{\lambda}(x,0))^2dx &\le& |\langle u^{\lambda}_{xx}+ u^{\lambda}_x,v^{\lambda}\rangle| \Big|_{x=0}^{x=L} +  |\langle u^{\lambda}_x-u^{\lambda},v^{\lambda}_x\rangle | \Big |_{x=0}^{x=L} \notag\\
&\le & \left(\|(u^{\lambda}_{xx}(0,.)\|_{H^1(0,T)'}+ \|u^{\lambda}_x(0,.)\|_{H^1(0,T)'}\right)\|v^{\lambda}(0,.)\|_{H^1(0,T)} \notag\\
&& + \left(\|(u^{\lambda}_{x}(0,.)\|_{H^1(0,T)'} + \|u^{\lambda}(0,.)\|_{H^1(0,T)'}\right)\|v^{\lambda}_x(0,.)\|_{H^1(0,T)} 
\end{eqnarray}
 and by Lemma \eqref{lemma inequality}
\begin{eqnarray*}
    \|u^{\lambda}(L,.)\|_{H^2(L,.)'}+ \|u_x^{\lambda}(L,.)\|_{H^2(0,L)'}+ \|u_{xx}^{\lambda}(L,.)\|_{H^2(0,L)'}&\le& C \|u^{\lambda}\|_{L^{\infty}(0,T,L^2(L,L+1)}\\
    &\le& C \|u^{\lambda}\|_{L^{\infty}(0,T,L^2(0,\infty))}.
\end{eqnarray*}
Moreover,  from the proof of  Lemma \eqref{soln} it is possible to conclude that  
$$v^{\lambda}(.,L), \ v_x^{\lambda}(.,L),\ v_{xx}^{\lambda}(.,L)\to 0 \ \text{as} \ L\to \infty \ \text{in} \  H^1(0,T) $$ 
hence, by \eqref{auxa}
\begin{gather}\label{contr}
\|v^{\lambda}(x,0)\|_{L^2(0,T)} \lesssim   \|\tilde{\Gamma}^{-1}\|\left( 
\|v^{\lambda}(0,.)\|_{H^1(0,T)} + \|v^{\lambda}_x(0,.)\|_{H^1(0,T)} \right),
\end{gather}
which makes a contradiction since the left side of the inequality \eqref{contr} is unbounded by the equation \eqref{normalema2} and the right side goes to $0$ when $\lambda \to 0$ by the equation \eqref{normalema1}.
\end{proof}

\subsection{Proof of the second main control result:} Theorem \eqref{control}.

We concentrate our attention on the controllability of the system \eqref{exact}:

 \begin{equation*}
     \begin{cases}
         \nu_t - \nu_{xxx}-\nu_{xx}=0, & \D'((0,\infty)\times (0,T)),\\
         \nu(x,0)=\nu_0,& x \in (0,\infty)
    \end{cases}
\end{equation*}
when the solution is locally supported in $L^2(\R^+, (0,T))$ satisfying $\nu(x,T)=\nu_T.$  Following the ideas of the Theorem 1.3 in \cite{Rosierub}, the proof can be summarized by transforming the problem \eqref{nonull} in an inhomogeneous initial value problem with homogeneous initial data driving into a null stage in time $T$ that is \eqref{auxw}, with the specification that the forcing term should belong to $L^2_{Loc}(\R^2)$.  In the proof of Claim \eqref{4.1R} it is generated a sequence of functions $\{u_n\}_{n \in\N}$ in $L^2(\square_n)$ defined over a bounded domain $(-n,n)$. 
when the solution is locally supported in $L^2(0,T,\R^+)$ satisfying $\nu(x,T)=\nu_T.$   Following the ideas of the Theorem 1.3 in \cite{Rosierub}, the proof can be summarized by transforming the problem \eqref{nonull} in an inhomogeneous initial value problem with homogeneous initial data driving into a null stage in time $T$ that is \eqref{auxw}, with the specification that the forcing term should belong to $L^2_{Loc}(\R^2)$.  In the proof of Claim \eqref{4.1R} it is generated a sequence of functions $\{u_n\}_{n \in\N}$ in $L^2(\square_n)$ defined over a bounded domain $(-n,n)$. 
%when the solution is locally supported in $L^2(0,T,\R^+)$ satisfying $\nu(x,T)=\nu_T.$ Following the ideas of the Theorem 1.3 in \cite{Rosierub}. The proof can be summarized by transforming the problem \eqref{nonull} in an Inhomogeneous initial value problem with homogeneous initial data driving into a null stage in time $T$ that is \eqref{auxw}, with the specification that the forcing term should belong to $L^2_{Loc}(\R^2)$. However, in the proof of Claim \eqref{4.1R} it is generated a sequence of functions $\{u_n\}_{n \in\N}$ in $L^2(\square_n)$ defined over a bounded domain $(-n,n)$. It is necessary to establish the convergence of this sequence to a function in $L^2(\R^+,[0,T])$, this will be obtained through the Approximation Theorem \ref{appthem}.  However, for each functions we will need to consider special properties which are summarized in  Proposition \ref{pobs}. \\

Initially, we need to rewrite the problem as a forcing problem with homogeneous stages defined in all $\R$.
The operator $A\nu=-\nu_{xxx}-\nu_{xx}$ with domain $\D(A)= H^3(\R^2)$ (or $L^2_b$ respectively), generates the continuous semi-group on $L^2(\R)$ (or $L^2_b$), which will be denoted by $\left( S(t) \right)_{t\ge 0}$. For $\epsilon>0$ and  $\tau \in (\epsilon, \frac{T}{2})$, consider the function $\varphi \in \mathcal{C}^{\infty}(\R)$ such that $\varphi\equiv 1$ for $t\le \tau$  and $\varphi \equiv 0$ for $t\ge T-\tau$.  For given $u_0\in L^2(0,\infty)$ and $u_T\in L^2_b$ defining 
    $$\nu(x,t)= \varphi(t) \nu_1(x,t) + (1-\varphi(t)) \nu_2(x,t) + \omega(x,t),$$
where $\nu_1$ is a solution of 
\begin{equation}\label{u1}
    \begin{cases}
    P\nu_1=0, &  \D'( \R\times(0,T)),\\
    \nu_1(x,0)=u_0(x),& \R,
    \end{cases}
\end{equation}
with $P=\partial_t +A$, that is $\nu_1(x,t)=S(t)\nu_0 $, and $\nu_2$ a solution of 
\begin{equation}\label{u2}
    \begin{cases}
    \frac{d}{dt}\nu_2=A\nu_2, &  \D'(\R\times(0,T)),\\
    \nu_1(x,T)=\nu_T(x),& \R,
    \end{cases}
\end{equation}
that is $\nu_2(x,t)=S(T-t)\nu_T(x)$ and $w$ satisfies 
\begin{equation}\label{auxw}
\begin{cases}
     P\omega= %\frac{d\varphi}{dt}S(t)(S(T)u_T-u_1),  
     \frac{\partial \varphi}{\partial t} \cdot \left( \nu_1-\nu_2\right) +2(\varphi(t)-1)A\nu_2, &  \D'(\R\times(0,T)),\\
     \omega(x,0)=\omega(x,T)=0, & x\in\R.
    \end{cases}
\end{equation}
Defining  $$f(x,t)= \frac{\partial \varphi}{\partial t} \cdot \left( \nu_1-\nu_2\right) +2(\varphi(t)-1)A\nu_2,$$ the function  $f\in L_{Loc}^2(\R^2)$, for $0<t_1<t_2<T$,  $Supp \ f \subset \R \times [t_1,t_2]$. For $\epsilon$ small enough i.e. $ 0<\epsilon < min \{t_1,T-t_2\},$  the function $u$ in \eqref{auxw} satisfies 
$$ Pw=f \ \text{in} \ \D'(\R^2) \ \text{and} \  Supp \ w \subset \R \times  [t_1-\epsilon,t_2+\epsilon] .$$

In effect, by the below Claim \eqref{4.1R}, % if $f?{}\in L^2(\R^2)$ with $Supp \ f \subset [t_1,t_2]\times \R$ and taking  $\epsilon \in (0,\min \{t_1,T-t_2\})$. Then there exists $u\in L^2_{loc}(\R^2)$ such that  $$Pu=f,\quad \D'(\R^2)\quad  \text{with}\quad Supp\ u \ \subset [t_1-\epsilon,t_2+\epsilon]\times \R,$$ and 
the proof is complete. $\square$

\begin{claim}\label{4.1R}
Let $f=f(x,t)$ be any function in $L^2_{Loc}(\R^2)$ such that  $Supp \,f \subset  \R \times [t_1,t_2]$ with $0<t_1<t_2<T$ and $\epsilon \in (0, \min \{t_1,T-t_2\}).$  Then there exists $u\in L^2_{Loc}(\R^2)$ such that  $$Pu=f \ \text{in} \ \D'(\R^2) \ 
 \text{with} \ Supp \ u \subset \R \times [t_1-\epsilon,t_2+\epsilon]. $$
\end{claim}
\begin{proof}    
Let us consider two different sequences of positive real numbers $\{t_{1,n}\}$ and $\{t_{2,n}\}$ such that for $n\ge 2$,
\begin{gather}\label{tseq}
    t_1-\epsilon < t_{1,n+1}< t_{1,n}< t_1<t_2<t_{2,n}<t_{2,n+1}<t_2 +\epsilon,
\end{gather}
for a sequence $\{ u_n\}_{n\ge 2}$ of functions such that, for each $n\ge 2 $ satisfies  $u_n \in L^2((-n,n)\times (0,T))$ and 
\begin{equation}\label{sequ}
    \begin{cases}
        Pu_n=f, \quad \square_{n}:=(-n,n)\times (0,T),\\
         Supp \, u_n \subset (-n,n) \times [t_{1,n},t_{2,n}],
    \end{cases}
\end{equation}
and $n>2$
\begin{equation}\label{key}
\| u_n -u_{n-1}\|_{L^2(\square_{n-2})}\le 2^{-n}.
\end{equation}
By Claim \eqref{CalermanForcing}, for $n\ge 2$  there exists a  subsequence, let us be called the same way, $\{u_n\}_{n\ge 2}$ satisfying  \eqref{sequ}, and also there exists a function $\omega\in L^2(\square_{n+1})$ such that 
\begin{equation*}
    \begin{cases} 
         P \omega =f, \quad  \square_{n+1},\\
        Supp \ \omega \subset (-(n+1),n+1) \times [t_{1,n+1},t_{2,n+1}].\\
    \end{cases}
\end{equation*}
Hence, $P(u_n-w)=0$ in  $\square_{n+1}$ and  $Supp\ (u_n-\omega) \subset (-n,n)\times [t_{1,n},t_{2,n}],$
with $t_{1,n+1}<t_{1,n}<t_{2,n}<t_{2,n+1},$ 
by Approximation Theorem \eqref{appthem},  there exists a function $v\in L^2(\square_{n+1})$ such that 
\begin{equation*}
    \begin{cases}
     Supp \ v \subset (-(n+1),n+1)\times [t_{1,n+1},t_{2,n+1}] ,\\
     P v =0, \quad  L^2(\square_{n+1}),   
    \end{cases}
\end{equation*}
and $$\| v-(u_n -\omega)\|_{\square_{n-1}}\le 2^{-(n+1)}.$$
Hence,  $u_{n+1}:= v+\omega$  satisfies \eqref{sequ}, making  a zero extension from $\square_{n-2}$ to $\square_n$, we conclude that  by inequality \eqref{key}, that the sequence 
$$\{u_n\} \xrightarrow{n\to \infty} u \quad  \text{in} \quad L^2_{loc}(\R^2)$$  and $Supp\ u \subset \R\times [t_2-\epsilon, t_2+\epsilon]$, additionally,  we conclude that  $Pu=f$ in $\R^2$.
\end{proof}

To fulfill the details of the previous proof, given that $\{u_n\}$ is defined over $\square_n=(-n,n) \times [0,T]$.  We will concentrate on analyzing the IBVP of the KdV-B.\\

\subsubsection{\bf Spectral analysis of the bounded KdV-B operator}\label{SL}

Considering the operator $A_Lv=v'''+ v''$ defined this time in $$\D_L(A):= \{ u \in H^3(-L,L): \partial_x^n(-L)=\partial_x^n(L), \ n=\text{0,1,2}\},$$
which generated the semi-group $S_L$  defined in $L^2(-L,L)$.  From the spectral analysis of the operator $A_L$ is obtained  the eigenvalues of $A_L$ are 
\begin{equation}\label{lambda}
    \lambda_n= -i \left(\frac{n\pi}{L}\right)^3-\left(\frac{n\pi}{L}\right)^2,
    \end{equation}
and the eigenfunctions of $A_L$ are  $e_n(x)=\frac{1}{\sqrt{2L}} e^{i  \frac{n \pi}{L}x}$ for $n\in \Z.$
Therefore,  any $u_0(x) \in L^2(-L,L)$ can be written as $u_0=\sum_{n\in\Z} c_ne_n$ for some $c_n\in \R$  and  for $(x,t)\in [-L,L]\times \R$,  $$S_L(t)u_0(x)=\sum_{n\in\Z}e^{ \lambda_n t}c_n e_n(x).$$\\

\begin{proposition}[Partial Observability result]\label{pobs}

    For any $0<l<L$, there exists a constant $C>0$ such that for every $u_0 \in L^2(-L,L)$, 
    \begin{equation}\label{ineq}
        \|u_0\|_{L^2(-L,L)}\le C \| u\|_{L^2\left((-l,l)\times (0,T)\right)}.
    \end{equation}
    Moreover,  \begin{equation}\label{ineq2}
        \|u\|_{L^2\left((-L,L) \times (0,T)\right)}\le \sqrt{T}C \| u\|_{L^2\left((-l,l)\times (0,T)\right)}.
    \end{equation}
    \end{proposition}
    \begin{proof}
        From the analysis of the spectrum of the operator $A_L$ given in \eqref{lambda}, the eigenvalues of $A_L$ are  $$\lambda_n=-i\left( \frac{n\pi}{L}\right)^3 -\left( \frac{n\pi}{L}\right)^2, \ \text{for} \ n\in \Z$$
        and the corresponding eigenfunctions are  \begin{equation}\label{en}
            e_n(x)= \frac{1}{2L} e^{i \frac{n \pi}{L}x}.
        \end{equation}

    $\left\{e_n(x)\right\}_{n\in \Z}$ is the ortho-normal Fourier basis in $L^2(-L,L)$.  Hence, for each $u_0\in L^2(-L,L)$ there exists a sequence of $c_n\in \C$ such that  $$u_0(x)= \sum_{n\in\Z} c_n e_n(x),$$
    consider  $a_n= e^{(\frac{n_0\pi}{L})^2T'} c_n$
    with $n_0\in \Z$ and $T'>0$ (to be selected later), therefore  for $\tu(x)= \sum_{n\in\Z} a_n e_n(x), $
    \begin{eqnarray*}
    \int_0^{T} \int_{-L}^L |u(x,\tau)|^2 dx d\tau &=& \int_0^{T} \int_{-L}^L|S_L(\tau)\tu(x)|^2 dx d\tau\\
    &=& \int_0^{T}  \int_{-L}^L \sum_{n\in \Z} |e^{(-i\left( \frac{n\pi}{L}\right)^3 -\left( \frac{n\pi}{L}\right)^2 ) \tau}  a_n e_n(x)|^2 dx d\tau\\
    &=&  \frac{1}{4L^2}\int_0^{T}  \int_{-L}^L \sum_{n\in \Z} |e^{(-i\left( \frac{n\pi}{L}\right)^3 -\left( \frac{n\pi}{L}\right)^2 ) \tau}   e^{(\frac{n_0\pi}{L})^2T'}  e^{i \frac{n \pi}{L}x} c_n |^2 dx d\tau
    %&=&  \frac{1}{4L^2}\int_0^{T}  \int_{-L}^L \sum_{n\in \Z} |e^{i\left( \frac{n\pi}{L}\right)^3 \tau} e^{-\left( \frac{n\pi}{L}\right)^2  \tau +   (\frac{n_0\pi}{L})^2T'}  e^{i \frac{n \pi}{L}x} c_n |^2 dx d\tau\\
    \end{eqnarray*}
   Taking $n_0=1$ and $T'=T$.  There exist $C_1,C_2 > 0$ depending on T such that   $  C_1\le e^{\frac{\pi^2}{L^2}\left( T-n^2 \tau  \right) } \le C_2 $ for $\tau \in [0,T].$
   Then 
   \begin{multline}\label{doble}
        \frac{C_1}{4L^2}\int_0^{T}  \int_{-L}^L \sum_{n\in \Z} |e^{-i\left( \frac{n\pi}{L}\right)^3 \tau}   e^{i \frac{n \pi}{L}x} c_n |^2 dx d\tau \le  \int_0^{T} \int_{-L}^L |u(x,\tau)|^2 dx d\tau \\ \le  \frac{C_2}{4L^2}\int_0^{T}  \int_{-L}^L \sum_{n\in \Z} |e^{-i\left( \frac{n\pi}{L}\right)^3 \tau}   e^{i \frac{n \pi}{L}x} c_n |^2 dx d\tau.
   \end{multline}
For $\alpha_n= \left( \frac{n \pi}{L}\right)^3$ with $n\in \Z$.  If $n\in \N$,
  \begin{gather*}
      \alpha_n-\alpha_{-n}= \frac{\pi^3}{L^3}\left( n^3-(-n)^3\right)=\frac{2\pi^3}{L^3} n^3> \frac{2\pi^3}{L^3}
  \end{gather*} 
  and if $-n\in \N, $
   \begin{gather*}
      \alpha_{-n}-\alpha_{n}= \frac{\pi^3}{L^3}\left( (-n)^3-n^3\right)=\frac{2\pi^3}{L^3} (-n)^3> \frac{2\pi^3}{L^3},
  \end{gather*} 
  therefore, for $\gamma:= \frac{2\pi^3}{L^3} $ 
  \begin{equation}\label{condIngham}
  \alpha_{n+1}-\alpha_n\ge \gamma > 0, \ \text{for all} \ n\in \Z.    
  \end{equation} 
  The Ingham's Inequality \cite[Theorem 1]{Ingham} implies that if  \eqref{condIngham} holds, for a fix $T^*> \frac{\pi}{\gamma}$ there exists a positive constant $C_3=C_3(T,\gamma)>0$ such that, for any finite sequence $\{b_n\}_{n\in \Z},$
  \begin{equation}\label{ingham}
  \sum_{n\in \Z}|b_n|^2 \le C_3\int_{-T^*}^{T^*} \left| \sum_{n\in \Z} b_n e^{-i \alpha_n t}\right|^2 dt.     
  \end{equation}
  for a fix $N_0\in \N$ large enough, let the sequence  $$b_n=\frac{a_n}{\sqrt{2L}} e^{i\left(\alpha_n T^*+
\frac{n\pi x}{L}\right)}, \ \text{ for} \ |n|\le N_0$$ and $b_n=0$ for $|n|> N_0$ the inequality \eqref{ingham} holds.\\

Consider    $\mathcal{Z}_n=Span\{ e_n : n\in \Z\}$, notice that $\mathcal{Z}:=\bigoplus_{n\in \Z} \mathcal{Z}_n\subset L^2(-L,L)$ and  the semi-norm for   $$p(u):= \left( \int_{-l}^l |u(x)|^2\right)^{1/2}, \ \text{for} \ u\in \mathcal{Z},$$ 
the inequality \eqref{ingham}
  \begin{eqnarray*}
  2l \sum_{n\in \Z} \frac{|a_n|^2}{2L} &\le &   C_{T^*}\int_{-l}^l\int_{0}^{2T^*} \left| \sum_{n\in \Z} a_n e^{-i \alpha_n \tau} e_n(x) \right|^2 d\tau dx\\
  &\le &  C_{T^*} \int_{0}^{2T^*} p\left( \sum_{n\le N_0 }  e^{-i \lambda_n \tau} a_n e_n(x) \right)^2 d\tau,
  \end{eqnarray*}
  Therefore,
%   \begin{gather*}  
%      \sum_{n\in \Z} |a_n|^2 \le    C_{T^*} \frac{L}{l} \int_{0}^{2T^*} p\left( \sum_{n\ge NN }  e^{-i \alpha_n \tau} a_n e_n(x) \right)^2 d\tau 
% \end{gather*}

\begin{gather*}  
     \sum_{n\le N_0} |c_n|^2 \le     \frac{LC_{T^*}}{l} \int_{0}^{2T^*} p(S_L(t)u_0^{N_0}(x))^2 d\tau, 
\end{gather*}
  with $C_{T^*}$ depending  only $T^*,L$ for $T>2T^*$ and $u_0^{N_0}(x)=\sum_{n\le N_0} c_n e_n(x)$, by arguments of density, the result holds for any $u_0 \in L^2(-L,L)$.\\
  
Therefore,  
\begin{equation*}
        \|u_0\|_{L^2(-L,L)}\le C \| u\|_{L^2\left( (-l,l) \times (0,T)\right)},
    \end{equation*}
    for $C= \frac{LC_{T^*}}{l}$.   From \eqref{doble}, we have 
  \begin{equation*}
        \|u\|_{L^2((-L,L)\times(0,T)}\le C_1 \| u\|_{L^2\left((-l,l)\times (0,T)\right)},
    \end{equation*}
    for $C_1=\frac{C_2 T}{4L^2}.$
   \end{proof}

The approximation theorem has been used in \cite{Rosay} and \cite{Rosierub} with different properties on the support.  We follow closely the ideas of \cite{Rosierub}, and by completeness, we present the proof.

 \begin{lemma}\label{appthem}{\bf Approximation  theorem}
For $n\in \N \smallsetminus \{0,1\} $ and $0<t_1<t_2<T$.  Let $u \in L^2((-n,n)\times (0,T) )$ be such that 
\begin{gather}\begin{cases}
    Pu=0,  \hspace{4cm} (-n,n)\times (0,T),\\
    Supp \ u \subset (-n,n)\times [t_1,t_2].
\end{cases}
\end{gather}
Let  $0 < \epsilon < \min \{t_1,T-t_2\}$. Then there exists $v \in L^2((-n,n)\times (0,T))$ such that 
\begin{gather}
    \begin{cases}
        Pv=0, \hspace{2.7cm} (-(n+1),n+1)\times (0,T) ,\\
    Supp \ v \subset (-(n+1),n+1) \times [t_1-\epsilon,t_2+\epsilon],
    \end{cases}
    \end{gather}
    and
    \begin{equation}\label{app}
        \|u-v\|_{L^2((-(n+1),n+1)\times (0,T))} < \epsilon.
    \end{equation}

\end{lemma}
\begin{proof}
    For $\eta>0$. Taking $L=n+1,$ by Lemma \eqref{approx} there exists $\tv \in  L^2((-(n+1),n+1)\times (0,T))$ such that 
        \begin{gather}
        \begin{cases}%%\label{u}
P\tv=0, &(-(n+1),n+1)\times (0,T), \\
\tv(.,t)= S_{n+1}(t-t_1+\frac{\epsilon}{2})v_1,& t\in (t_1-\frac{\epsilon}{2},  t_1-\frac{\epsilon}{4}),\\
\tv(.,t)= S_{n+1}(t-t_2- \frac{\epsilon}{4})v_2,& t\in (t_2+\frac{\epsilon}{4},  t_2+\frac{\epsilon}{2}).
\end{cases}
    \end{gather}
for some $v_1,v_2 \in L^2(-(n+1),n+1)$
    and $$\|\tv-u\|_{L^2(\square_{\frac{\epsilon}{4}})}<\eta.$$
    Considering the cut-off function $\varphi\in\D(0,T)$ such that  $0\le \varphi \le 1$, $\varphi\equiv 1$ in $[t_1-\frac{\epsilon}{4},t_2+\frac{\epsilon}{4}]$ and $Supp \ \varphi \subset [t_1-\frac{\epsilon}{2},t_2+\frac{\epsilon}{2}].$ Therefore, if $$\bv(x,t)= \varphi(t) \tv(x,t),$$ hence,
    $$Supp \ \bv \subset (-(n+1),n+1)\times [t_1-\frac{\epsilon}{2},t_2+\frac{\epsilon}{2}]. $$
    Moreover,  since $Supp \ u \subset (-n,n)\times [t_1,t_2] $ and $\varphi(t)\equiv 1$ in $[t_1-\frac{\epsilon}{4},t_2+\frac{\epsilon}{4}]$, we obtain
    \begin{eqnarray}\label{appp}
        \|\bv -u\|_{L^2((-(n+1),n+1) \times (0,T))} & = & \|\bv - u \|_{L^2(\square_{\frac{\epsilon}{4},n-1})} \notag\\
        &\le&  \| (\varphi-1)\bv \|_{L^2(\square_{\frac{\epsilon}{4},n-1})} + \| \tv - u \|_{L^2(\square_{\frac{\epsilon}{4},n-1})} \notag\\
        &\le& \| \tv \|_{L^2((-(n-1),n-1)\times  \{(t_1-\frac{\epsilon}{2},t_1-\frac{\epsilon}{4})\cup  (t_2+\frac{\epsilon}{4},t_2+\frac{\epsilon}{2}) \})} + \| \tv - u \|_{L^2(\square_{\frac{\epsilon}{4},n-1})} \notag\\
        &\le& 2\|\tv-u\| _{L^2(\square_{\frac{\epsilon}{4},n-1})} \notag \\
        &\le& 2\eta. 
    \end{eqnarray}
 Taking into account that   $$P\bar{v} = \frac{d\varphi}{dt}\tv, \quad  (-(n+1),n+1)\times(0,T) ,$$ and the construction of $\tv$
    
    \begin{gather*}
        \|P\bar{v} \|^2_{L^2((-(n+1),n+1) \times (0,T) )} \le \|\frac{d\varphi}{dt}\tv\|^2_{\infty} \|\tv\|^2_{L^2((-(n-1),n-1)\times \{(t_1-\frac{\epsilon}{2},t_1-\frac{\epsilon}{4})\cap  (t_2+\frac{\epsilon}{4},t_2+\frac{\epsilon}{2}) \} )}.
    \end{gather*}
    From Proposition \ref{pobs}, there exists a constant $C>0$ depending only on $\epsilon$ and $n$ such that
    \begin{gather*}
    \|\tv\|_{L^2( (-(n+1),n+1) \times  (t_1-\frac{\epsilon}{2},t_1-\frac{\epsilon}{4})) } \le C \|\tv \|_{L^2((-n+1,n-1) \times  (t_1-\frac{\epsilon}{2},t_1-\frac{\epsilon}{4})) },\\
    \|\tv\|_{L^2((-(n+1),n+1) \times  (t_2+\frac{\epsilon}{4},t_2-\frac{\epsilon}{2})) } \le C \|\tv \|_{L^2((-n+1,n-1) \times  (t_2-\frac{\epsilon}{4},t_2-\frac{\epsilon}{2})) }.\\
\end{gather*}
    Therefore,  from \eqref{appp}
      \begin{equation}\label{apder}
         \|P\bar{v} \|^2_{L^2((-(n+1),n+1) \times (0,T) )} \le  C \eta \|\frac{d\varphi}{dt}\tv\|_{\infty}.
      \end{equation}
By proposition \eqref{CalermanForcing}, there exist a constant $C^*>0$ depending on $t_1,t_2,n,\epsilon$ and a function $w\in L^2( (-(n+1),n+1) \times (0,T))$ such that 
    \begin{gather*}
        Pw= P \Bar{v} \quad (-(n+1),n+1) \times (0,T), \\
        Supp \ w \subset (-(n+1),n+1) \times [t_1-\epsilon, t_2+\epsilon] 
    \end{gather*}
    and   
     \begin{equation}\label{appw}
         \|w \|_{L^2((-(n+1),n+1) \times (0,T) )} \le  C^* \eta \| P\Bar{v}\|_{L^2((-(n+1),n+1) \times  (0,T))}.
    \end{equation}
    to conclude the proof we consider $\eta$ small enough and  select $v=\Bar{v}-u$, then  from \eqref{appp}, \eqref{apder} and \eqref{appw}, we obtain 
    \begin{equation}
        \|v-u\|_{L^2((-(n+1),n+1)\times (0,T) )} \le \left( 2 + C C^* \| \frac{d\varphi}{dt}\|_{\infty} \right)\eta.
    \end{equation} \end{proof}

\appendix

 \section{Useful Results}\label{useful results}
Pseudodifferential operators are a powerful generalization of differential operators that enables a wider range of operations on functions.  A symbol is essentially a smooth function that characterizes the operator's behavior. More formally, for a given real number $m \in \mathbb{R}$, the symbol class $S^m$ is defined as follows
 \begin{equation}\label{simbol}
    S^m=\left\{p(x, \xi) \in C^{\infty}\left(\mathbb{R}^n \times \mathbb{R}^n\right):|p|_{S^m}^{(j)}<\infty, j \in \mathbb{N}\right\},
\end{equation}
where

$$
|p|_{S^m}^{(j)}=\sup \left\{\left\|\langle\xi\rangle^{-m+|\alpha|} \partial_{\xi}^\alpha \partial_x^\lambda p(\cdot, \cdot)\right\|_{L^{\infty}\left(\mathbb{R}^n \times \mathbb{R}^n\right)}:|\alpha+\lambda| \leq j\right\}
$$

and $\langle\xi\rangle=\left(1+|\xi|^2\right)^{1 / 2}$.

\begin{theorem}\label{Sobolev boundedness} (Sobolev boundedness). Let $m, s \in \mathbb{R}$ and $p \in S^m$. Then the pseudodifferential operator $\Psi_p: \mathcal{S}(\mathbb R)\to \mathcal{S}(\mathbb R)$ associated to the symbol $p$, defined by
$$
\Psi_p f(x)=\frac{1}{2\pi i}\int_{\mathbb{R}} e^{ i x \cdot \xi} p(x, \xi) \widehat{f}(\xi) d \xi, \quad f \in \mathcal{S}\left(\mathbb{R}^n\right)
$$
extends to a bounded linear operator from $H^{m+s}\left(\mathbb{R}^n\right)$ to $H^s\left(\mathbb{R}^n\right)$. 
\end{theorem}

 \begin{lemma}\label{aux lemma}
    Suppose $0<v<1$ and let I be the integral operator defined by
$$[I(f)](t):=\int_{-\infty}^{\infty} e^{i |x| \eta+\mu(\eta)|x|} f(\eta) d \eta, \quad \text { for } x \in \mathbb R,
$$
where $\mu: \mathbb{R} \rightarrow(-\infty, 0]$ is a continuous function satisfying
$$
|\mu(\eta)| \sim \eta^\nu \quad \text { as } \eta \rightarrow 0 \quad \text { and } \quad|\mu(\eta)| \sim|\eta|^{\frac{1}{3}} \quad \text { as }|\eta| \rightarrow \infty .
$$

Then, $I$ is a bounded linear operator on $ L^2(\mathbb{R})$, which is to say, there exists a constant $C$ such that
$$
\|I(f)\|_{L^2(\mathbb{R})} \leqslant C\|f\|_{L^2(\R)}
$$
for all $f \in  L^2(\mathbb{R})$.
\end{lemma}
The proof of this Lemma is a minor modification of the proof of Lemma 2.6 in \cite{BSZ08}.

%\begin{lemma} \label{bound 1} For $\varphi\in C_0^\infty(\mathbb R)$ and all $n=0,1,2\cdots,$ there exists a constant $C=C(n,\varphi) > 0$ such that
   % \begin{align*} 
   %I:= \int_{\mathbb R^3}\langle i(\tau+\xi^3)+\xi^2\rangle^{2b}  |\widehat{\varphi}(\tau-\tau_1)|^2 \frac{|\tau_1|^{\frac{2n}{3}}}{(\tau_1^{\frac{2}{3}}+\xi^2 )^2}d\tau_1 d\tau d\xi \leq C
%\end{align*}
%\end{lemma}
%The proof follows ... \textcolor{red}{completar}
%\input{Viejitos/apendiceC}  Bourgain estimates

% \begin{lemma}\label{roots}
% For $\tau \in \R$ \textcolor{red}{ in here $\tau$ should not be Re $\tau>0$? for the definition of the roots??}
%      the roots of the polynomial $\tau-r^3-r^2=0$ as in \eqref{raices}
%      $$r_j(i\tau_1)=O(\tau_1^{-\frac{1}{3}})$$ 
% \end{lemma}
% \begin{proof}
%     Completar
% \end{proof}

 \section{Lemmas for the no-controllability}
The two following lemmas are required for completeness of the proof in Theorem \eqref{nonullcontrol} which follows the ideas in \cite[Lemma 2.1]{Rosierub} for the KdV-B operator of the IBVP in \eqref{1.1}.
\begin{lemma}\label{soln}
    There exists a family of functions $$\left(v^{\lambda}\right)_{\lambda>0} \in \bigcap_{n\ge0} C^{\infty}([0,T],H^n(\R^+)),$$
    such that for every $\lambda>0$, $v_{\lambda}$ solves  
    \begin{gather}\label{lemavl}\begin{cases}
        (\partial_t - \partial_{xxx}-\partial_{xx})\vl =0, & (0,\infty)\times (0,T),\\
        \vl \big|_{x=0}=0,& (0,T),
        \end{cases}
    \end{gather}
and for all  $n\ge 1$,
\begin{gather}%\label{normalema}
\|\vl_{x} \big|_{x=0}\|_{H^n(0,T)}+ \| \vl_{xx} \big|_{x=0}\|_{H^n(0,T)} \to 0 \quad \text{as} \quad \lambda \to 0, \label{normalema1}\\
  \|w^{\lambda}(x,0)\|_{L^2(\R^+)}^2 \quad \text{is \quad unbounded}. \label{normalema2}
\end{gather}
\end{lemma}
\begin{proof}
%%%
Considering the operator $A \nu:=  (\partial_{xxx} + \partial_{xx})\nu$ defined in $\D(A)= H^3(0,+\infty)\cap H_0^1(0,\infty) \subset L^2(0,\infty)$, which generates a strongly continuous semigroup $(S(t))_{t>0}$, as in \eqref{semisol}.  Let us consider for any $\lambda>0$,  the solution of 
$$\partial_t v^{\lambda} = A v^{\lambda},$$ in an exponential form
$$v^{\lambda}(x,t)= S(t)v_0^{\lambda}(x)= e^{-\lambda t} v_0^{\lambda}(x),  \quad \quad (x,t)\in (0,\infty) \times (0,T),$$
where $v_0^{\lambda}\in \D(A)$ solves the eigenvalue problem 
$$ Av_0^{\lambda} = - \lambda v_0^{\lambda},  \quad \lambda>0.$$

 For every $\lambda>0$ fixed, the roots $z_j$ with  $j=1,2,3$ of the polynomial   $z^3+z^2=- \lambda$ cannot be all Reals.  In effect, if there are real    
 for each $j=1,2,3$, $z_j^2(z_j+1)=-\lambda < 0$ so $z_j^2>0$.  $z_j$  also should satisfies $z_j<-1$,  which it is not possible since $z_1z_2z_3<0$. Therefore, consider one real root, $z_1$, and two complex roots $z_2=a+ib,z_3=\Bar{z_2}$, in this scenery by the Newton-Girard relations we roots should satisfy 
\begin{gather*}
    z_1z_2+z_1z_3+z_2z_3=0,\\
    z_1z_2z_3=-\lambda,\\
    z_1+z_2+z_3=1,
\end{gather*}
hence,  $b^2=a(2+3a)>0$ and $\lambda = 2a(1+3a)(2a+1)>0$ so,  $a>0$,  from where we conclude that  
\begin{gather*}
z_1= -(1+2a), \quad z_2= a+ i b,  \quad z_3= a-ib. 
\end{gather*}

%If we consider the solution $ w^{\lambda}(x,t)=  e^{-\lambda t}e^{-(1+2a)x} $ satisfies 
%\begin{itemize}
 %   \item $P w^{\lambda}=0$ in $(0,T)\times (0,\infty)$
    %\item $w^{\lambda}(0,t)=e^{-%\lambda t}$  on $(0,T)$
    %\item $\|w^{\lambda}(x,0)\|_{L^2(\R^+)}^2=\int_{\R^+} e^{-2(1+2a)x}dx = \displaystyle\frac{1}{2(1+2a)}$
    %\item $\|w^{\lambda}_x (0,t)\|_{H^n(0,T)}+ \|w^{\lambda}_{xx} (0,t)\|_{H^n(0,T)} = 
    %\|w^{\lambda}_x (0,t)\|_{H^n(0,T)}+ \|w^{\lambda}_{xx} (0,t)\|_{H^n(0,T)} \to 0$ as $\lambda\to 0$, for $n\ge 1$. 
%\end{itemize}

If we consider a solution 
$$ w^{\lambda}(x,t)=  e^{-\lambda t}Im(e^{(a+i b)x})=e^{-\lambda t} e^{a x} \sin{b x} $$ it satisfies 
\begin{enumerate}
    \item[i.] $P w^{\lambda}=0$ in $(0,T)\times (0,\infty)$
    \item[ii.] $w^{\lambda}(0,t)=0 $  on $(0,T)$
    \item[iii.]  $  \|w^{\lambda}(x,0)\|_{L^2(\R^+)}^2$ is unbounded. In effect, 
    \begin{eqnarray*}
    \|w^{\lambda}(x,0)\|_{L^2(\R^+)}^2&=&\int_{\R^+} e^{2a x}\sin^2{(b x)} dx\\
    &=&  \displaystyle\frac{e^{2 a x} \left(a^2 (-\cos (2 b x))+a^2-a b \sin (2 b
   x)+b^2\right)}{4 a \left(a^2+b^2\right)}\Big|_0^{\infty}=\infty\end{eqnarray*}
    \item[iv.] $\lim_{\lambda\to 0}\|w^{\lambda}_x (0,t)\|_{H^n(0,T)}+ \|w^{\lambda}_{xx} (0,t)\|_{H^n(0,T)}=0$. En effect, 
    $\|w^{\lambda}_x (0,t)\|_{H^n(0,T)}+ \|w^{\lambda}_{xx} (0,t)\|_{H^n(0,T)} = 
    \|b e^{-\lambda t}\|_{H^n(0,T)}+ \|2abe^{-\lambda t}\|_{H^n(0,T)} \to 0$ as $\lambda\to 0$. Notice that $\lambda \to 0$ is equivalent to $a,b\to 0$, for $n\ge 0$ 
\end{enumerate}
\end{proof}
%let us define $v^{\lambda}(x,t):=\displaystyle \frac{w^{\lambda}}{\|w^{\lambda}\|_{L^2(\R^+)}}.$

The following lemma analyzes the traces of the bounded energy solution of \eqref{nonull}. The proof is presented for the sake of completeness and it is an adaptation of \cite[Lemma 2.2]{Rosierub}.
\begin{lemma}\label{lemma inequality}%used for the proof of no null controllability result
    For $T>0$ and $L>0$ given numbers and let $u\in L^{\infty}(0,T, L^2(0,L))$ such that $$u_t-u_{xxx}-u_{xx}=0 \  \text{in} \  \D'((0,L)\times (0,T)).$$  Then $u\in H^3(0,L, H^1(0,T)')$ and there exists $C=C(L,T)>0$ such that
    \begin{eqnarray}\label{dual}
        \|u(0,.)\|_{H^1(0,T)'}+\|u_x(0,.)\|_{H^1(0,T)'}+\|u_{xx}(0,.)\|_{H^1(0,T)'} \le C\|u\|_{L^{\infty}(0,T;L^2(0,L))}.
    \end{eqnarray}
\end{lemma}
\begin{proof}
   Notice that  $u_t=u_{xxx}+u_{xx}\in L^2(0,T;H^{-3}(0,L))$, therefore
    $u\in H^1(0,T;H^{-3}(0,L)).$ By integration by part 
    \begin{equation}\label{aux1}
        \int_0^T \langle u_t,f\rangle dt= \langle u, f\rangle \Big|_{t=0}^T- \int_0^T \int_0^L u f_t dx dt,
    \end{equation}
    for all $f\in H^1(0,T; H_0^3(0,L))$, with the inner product is defined on $H^{-3}(0,L)\times H^3_0(0,L)$.\\
    
    Since $u\in C([0,T];H^{-3}(0,L))\cap L^{\infty}(0,T,L^2(0,L))$, $u$ is weakly continuous in $L^2(0,L)$. From \eqref{aux1} there exists $C=C(T,L)$ such that
    \begin{eqnarray*}%\label{aux2}
       \Big| \int_0^T \langle u_t,f\rangle dt \Big| 
       &\le&  \|u\|_{ L^{\infty}(0,T,L^2(0,L))} \left( \|f(.,T)\|_{} + \|f(.,0)\|_{} + \|f\|_{L^2(0,L; H^1(0,T))}\right)\\
       &\le & C \|u\|_{ L^{\infty}(0,T,L^2(0,L))} \|f\|_{L^2(0,L, H^1(0,T))},
    \end{eqnarray*}
    together with  $\overline{H^1(0,T,H^3_0(0,L))} = L^2(0,L,H^1(0,T))$. It can be concluded that $u_t \in L^2(0,L, H^1(0,T)'),$
and
     $u_t=u_{xxx}+u_{xx}$ in $\D'((0,T)\times (0,L))$.  Since, $u_t\in L^2(0,L,H^1(0,T)')$, we have $u\in H^3(0,L,H^1(0,T)')$, so \eqref{dual} holds.\end{proof}

\section{Auxiliar Results for the analysis of controllability:}

Considering Carleman's estimate,  there exists a distributional solution on $\R\times(-L,L)$ for the non-linear KdB-Burger equation \eqref{1.1} with $a=1$ and the non-homogeneous function $f:=F(x,t)$.

\begin{claim}\label{CalermanForcing}
    For $L>0$ and for any function $f\in L^2((-L,L))\times\R$ such that for  $-\infty < t_1<t_2 < \infty,$ Supp$f \subset (-L,L) \times [t_1,t_2]$.  For any $\epsilon > 0$ there exist $C=C(L,t_1,t_2,\epsilon)$ and a function $v\in L^2((-L,L)\times\R)$ such that 
    \begin{gather}\label{eq*}
            v_t-v_{xxx}-v_{xx}=f,  \quad \D'( (-L,L)\times \R),\\
            \text{Supp} \ v \subset    (-L,L)\times[t_1-\epsilon, t_2+\epsilon],\\
            \|v\|_{L^2(-L,L)\times \R)}\le C \|f\|_{L^2( (-L,L)\times \R)}.
    \end{gather}
\end{claim}
\begin{proof}
    The proof follows some ideas in \cite[Corollary 3.2]{Rosierub}, for the full operator of the KdV equation, which is skew-adjoint, however, the proof has been  made for the KdV-Burger equation which also includes the term $-\partial_{xx}u$, which is self-adjoint.\\
  
 Supposing that $0=t_1-\epsilon < t_1<t_2< t_2 + \epsilon=: T$, by the Carleman's estimate  \eqref{Carleman} 
    \begin{gather*}
        \int\limits_{\SLT}  \displaystyle\frac{s^5  }{t^5(T-t)^5} q^2 e^{-2s \psi} dxdt\le \int\limits_{\SLT}  C ( q_t- q_{xx}-q_{xxx})^2 e^{-2s \psi}dxdt,
    \end{gather*}
taking 
\begin{equation*}
    C_{s,T}= \min_{[0,T]} \frac{s^5}{t^5 (T-s)^5}, \quad K_s= \max_{[-L,L]} 2s \phi(x) \quad \text{and} \quad K_{s,T}=\max_{[0,T]} {-\frac{-k_s}{t(T-t)}}, \end{equation*} 
   then \begin{gather}\label{*}
        \int\limits_{\SLT}  \displaystyle q^2 e^{-\frac{K_s}{t(T-t)}} dxdt\le \int\limits_{\SLT}  C_1(Pq )^2 dxdt,
    \end{gather}
    with $C_1=C k_{s,T}$ and $Pq=q_t- q_{xx}-q_{xxx}$.  Notice that the operator $P$ can be decomposed in a skew-adjoint operator $P_1 q= q_t-q_{xxx}$ and a self-adjoint operator $P_2 q= q_{xx}$,  with $P:=P_1+P_2$.\\  
    
    Defining the bilinear forms 
    \begin{gather}\label{inner}
    (p,q):= \int\limits_{\SLT}  Pp Pq dxdt
    \end{gather}
    which is a scalar products on $\Upsilon$ in \eqref{Uset}, we will call  $\h$  to the completition spaces of $\Upsilon$ with the defined inner products.\\

    We can conclude that if $q \in \h $  then \eqref{*} holds,  defining for each $f \in L^2((-L,L)\times \R)$ such that $\supp\,f \subset (-L,L)\times [t_1,t_2] $ the continuous linear forms  defined in $\h$ as
    $$\mathrm{l}(q)= -\int\limits_{\SLT} f qdxdt,$$
     we observe that,
    $$|\mathrm{l}(q)|=\int\limits_{\SLT} |f q|dxdt \le 
    \|f\|_{L^2(\SLT)} \int\limits_{\SLT} q^2 dxdt \le
    \|f\|_{L^2(\SLT)} \int \limits_{\SLT} (P q)^2 dxdt= \|f\|_{L^2(\SLT)} (q,q)^{\frac{1}{2}}$$
    since the hypothesis of the Riesz Representation Theorem is fulfilled, there exists a unique $ p\in \h$ such that for all $ q\in \h, \quad \mathrm{l}(q)= (p,q)$.\\

    Let us call $\nu := Pp \in L^2(\SLT)$ and  for any test function $q\in \D(\SLT)$ (denoting as well $\D':=\D'(\SLT)$), we obtain
    \begin{eqnarray*}
    \mathrm{l}(q)&=&(p,q)_{\D'\times\D}=  \int\limits_{\SLT} Pq Pq dxdt 
    =\int\limits_{\SLT} (P_1p P_1q + P_2p P_2q) dxdt  
    %&=&=\int\limits_{\SLT} vPq dxdt=\int\limits_{\SLT} (-P_1 v + P_2 v )q dxdt\\
    %&=&\int\limits_{\SLT} -fq dxdt + \int\limits_{\SLT} 2P_2 vqdxdt - \langle v,q\rangle|_{t=0}^{t=T},
\end{eqnarray*}
    observing that $P_1^*=-P_1$ and $P_2^*=P_2$ operators that commute between them, we have that 
    $$\int\limits_{\SLT} (P_1p P_2q + P_2p P_1q) dxdt=\int\limits_{\SLT} (P_2P_1 - P_1 P_2)p q dxdt=0.$$
therefore, since $P\nu=f$ in $\D'$ and $\nu  \in H^1(0,T, H^{-3}(-L,L))$, hence
$\nu_t=f+\nu_{xx}+\nu_{xxx}  \in L^2(0,T, H^{-3}(-L,L))$. For a $q\in \Upsilon \subset H^1(0,T,H_0^3(-L,L))$, we have

 \begin{eqnarray*}
    \mathrm{l}(q)&=&(p,q)_{\D'\times\D}=  \int\limits_{\SLT} Pq Pq dxdt 
    =\int\limits_{\SLT} (P_1p P_1q + P_2p P_2q) dxdt \\ 
    &=&\int\limits_{\SLT} vPq dxdt=\int\limits_{\SLT} (-P_1 v + P_2 v )q dxdt - \langle v,q\rangle|_{t=0}^{t=T}\\
    &=&\int\limits_{\SLT} -fq dxdt + \int\limits_{\SLT} 2P_2 vqdxdt - \langle v,q\rangle|_{t=0}^{t=T},
\end{eqnarray*}
  we obtain that for any $q\in \D((-L,L)\times (0,T))$,  $\int\limits_{\SLT} 2P_2 vqdxdt = \langle v,q\rangle|_{t=0}^{t=T}$  equivalently to  $-\int\limits_{\SLT} 2 (\nu_x)^2 dxdt= \langle \nu, \nu \rangle \Big|_0^T$, therefore  $v\Big|_{t=0}=v\Big|_{t=T}=0 \in H^{-3}(-L,L)$.  Extending $\nu$ by zero in the complement of $\SLT$ the proof is complete.\end{proof}

The following lemmas will help us to study the convergence of the sequence $u_n$ in $L^2(\R^+)$ using a version of the approximation Theorem.  Those are given for the sake of completeness. However, they adapt the results on \cite[lemma 4.2]{Rosierub}. \\

\begin{lemma}\label{approx}
    For $0<l_1<l_2<L$ and $0<t_1<t_2<T$ real numbers.  Let $u\in L^2((-l_2,l_2)\times (0,T))$ such that 
    $$Pu=0 \quad (-l_2,l_2)\times (0,T)$$
    and $Supp \ u \subset (-l_2,l_2)\times [t_1,t_2].$  Let $\delta>0$ such that $2\delta < min \{ t_1, T-t_2 \}$ and $\eta>0$ then there exist $v_1, v_2 \in L^2(-L,L)$ and $v\in L^2(\square)$ such that  
    \begin{gather}
        \begin{cases}\label{u}
Pv=0, &(-L,L)\times (0,T), \\
v(.,t)= S_L(t-t_1+ 2 \delta)v_1,& t\in (t_1-2\delta,  t_1-\delta),\\
v(.,t)= S_L(t-t_2- \delta)v_2,& t\in (t_2+\delta,  t_2+2\delta).
\end{cases}
    \end{gather}
    and $$\|v-u\|_{L^2(\square_{\delta})}<\eta.$$
\end{lemma}
\begin{proof}

    The function set 
    $$\mathcal{E}= \left\{ v\in L^2(\square): \text{there exist} \ v_1,v_2\in L^2(-L,L) \  \text{s.t.} \ \eqref{u} \, \text{is satisfied} \right\}$$ 
    For $u \in \mathcal{E}^{\perp \perp}=\Bar{\mathcal{E}}$ with respect to $L^2(\square_{\delta})$ and $g\in \mathcal{E}^{\perp}$,   $$(u,g)_{L^2(\square)}=0.$$
    
    Taking $u\in L^2((-l_2,l_2) \times (0,T))$ such that 
    $$Pu=0, \quad (-l_2,l_2)\times (0,T),$$
    and $Supp \ u \subset  (-l_2,l_2)\times [t_1,t_2].$  Considering an approximation of $u$, such that  for $\eta >0$,  the function  $\nu$ is the mollified of  $u$, $\nu:=\nu^{\eta} \in  \mathcal{D}(\R^2)$,   satisfying $$P\nu=0, \quad (-l_1,l_1)\times(0,T),$$
    $Supp \ \nu \subset (-l_2,l_2)\times [t_1-\delta,t_2+\delta]$  and  $\|\nu - u\|_{L^2((-l_1,l_1)\times(0,T))}\le \frac{\eta}{2}.$\\
We would like to prove that  $$(\nu,g)_{L^2(\square_{\delta})}=0.$$
    
Since $g\in \mathcal{E}^{\perp}\subset  L^2(\square_{\delta})$ and for  $\omega \in L^2(\square_{\delta})$,  $E(g)$ and $E(\omega)$ are the zero extensions on $\R^2$ respectively.  Let $\Omega_{\delta}$  the strip  $\R \times (t_1-\delta,t_2+ \delta)$, for any $\varphi \in \D(\Omega)\subset 
 \mathcal{X}:=\left\{ \varphi \in \mathcal{C}^{\infty}(\R^2): \ Supp \ \varphi \subset \R \times [t_1- \delta, t_2 +\delta]  \right\}.$
\begin{gather*}
    (\varphi,g)_{L^2(\square_{\delta})}=(\varphi,E(g))_{L^2(\Omega_{\delta})},\\
     (P\varphi,\omega)_{L^2(\square_{\delta})}=(P\varphi,E(\omega))_{L^2(\Omega_{\delta})},
     \end{gather*}
     by \eqref{claim2eq},  
        \begin{equation*}
        (\varphi,g)_{L^2( (-l_1,l_1)\times  (t_1-2\delta, t_1+2\delta))}=(P\varphi,w)_{L^2((-L,L)\times (0,T) }.
    \end{equation*}   
we have that  $$\langle P^* E(\omega), \varphi \rangle_{\D'(\Omega_{\delta}),\D(\Omega_{\delta})}=\langle E(g), \varphi \rangle_{\D'(\Omega_{\delta}),\D(\Omega_{\delta})},$$
therefore, $$P^* E(\omega)=E(g) \quad \text{in} \quad \D'(\Omega_{\delta})$$
and $P^* E(\omega)=0$ for $t\in (t_1-\delta, t_2+\delta)$ and $|x|>L$, by Holmgren's uniqueness theorem \cite[Theorem 8.6.5]{Holmgren5}, $P^* E(\omega)=0$ for $t\in (t_1-\delta, t_2+\delta)$ and $|x|>l_1$, using this argument for $\nu \in \mathcal{X}$ (see \eqref{claim1},
\begin{eqnarray*}
    (\nu,g)_{L^2(\square_{\delta})}&=&(P\nu,w)_{L^2(\square)}\\
    &=&(P\nu,w)_{L^2((-l_1,l_1) \times (t_1-\delta,t_2+\delta))}\\
    &=&0. 
\end{eqnarray*} \end{proof}

\begin{claim}\label{claim1}%[Requiered for Lemma \eqref{lemma inequality}] %Claim 1
    For the function set 
    $$\mathcal{X}=\left\{ \varphi \in \mathcal{C}^{\infty}(\R^2): \ Supp \ \varphi \subset \R \times [t_1- \delta, t_2 +\delta]  \right\}.$$
    Then, there exists $C>0$ such that for $0<l_1<l_2<L$,  $0<t_1<t_2<T$ and for all $\varphi \in \mathcal{X}$,
    \begin{equation}\label{claim1eq}
        |(\varphi,g)|_{L^2( (-l_1,l_1)\times (t_1-2\delta, t_2+2\delta))} \le C \|P \varphi \|_{L^2((-L,L)\times (0,T) )}.
    \end{equation}
\end{claim}
\begin{proof}
    Let us define  for each $\varphi \in \X$ and $t \in [0,T]$
    \begin{gather*}
        \nu(t):= \int_0^t  S_L(t-\tau) P\varphi(\tau) d\tau,
    \end{gather*}
   which is the strong solution of  
   \begin{equation*}
       \begin{cases}
           P\nu=P\varphi, & (x,t)\in[-L,L]\times[0,T]\\
           \partial_x^j\nu(-L,t)=\partial_x^j\nu(L,t),& n=0,1,2\\
           \nu(.,0)=0.
       \end{cases}
   \end{equation*}
   The function $\nu - \varphi \in \X$ then for any $g \in \ X^{\perp}\subset L^2((-l_1,l_1)\times(t_1-2\delta, t_2+2\delta) )$, $$(\nu-\varphi,g)=0.$$
   Moreover, for any $t\in[0,T]$, from the definition of $\nu$ and H\"older inequality
   $$\|\nu(t)\|_{L^2(-L,L)}\le \|P\varphi\|_{L^1(0,t,L^2(-L,L))}\le \sqrt{T}\|P\varphi \|_{L^2([-L,L]\times[0,T])},$$
   and we conclude that 
   $$|(\varphi,g)| =|(\nu,g))| \le \sqrt{T} \|g\| \|\nu\|\le T \|g\| \|P\varphi\|_{L^2([-L,L]\times[0,T])}$$
   the norms and inner products no indicated are defined in $L^2((-l_1,l_1)\times (t_1-2\delta, t_2+2\delta))$.
\end{proof}

\begin{claim}%[Requiered for Lemma \eqref{approx}] %Claim 2
    There exists a function $w\in L^2((-L,L)\times (0,T) )$ such that for all $\varphi \in \mathcal{X}$,
    \begin{equation}\label{claim2eq}
        (\varphi,g)_{L^2((-l_1,l_1)\times (t_1-2\delta, t_1+2\delta))}=(P\varphi,w)_{L^2((-L,L)\times (0,T)}.
    \end{equation}    
\end{claim}
\begin{proof}
    Considering the set
    $$\mathcal{Z}:= \left\{ P\varphi\Big|_{[-L,L]\times[0,T]}: \varphi \in  \X \right\}.$$
For any  $\xi \in \mathcal{Z}$, such that $\xi=P\varphi_1|_{\square}=P\varphi_2|_{\square},$ with $\varphi_1,\varphi_2\in\X$, we have that for $g\in \mathcal{E}^{\perp}, \quad$   $\varphi_1-\varphi_2 \in \mathcal{E}$, i.e.
$$(\varphi_1-\varphi_2,g)_{L^2(\square_{\delta})}=0.$$
Let us consider, the linear map $\Lambda: \mathcal{Z}\to \R,$ defined  for any $\xi \in \mathcal{Z},$ i.e. $\xi=P\varphi\Big|_{\square}$ for $\varphi\in \X$ and $g\in \mathcal{E}^{\perp}$, as before.  It is possible to conclude that  $\Lambda(\xi)=(\varphi,g)_{\square_{\delta}}$,   is well defined.\\

Holding $H$ the closure of $\mathcal{Z}$ in $L^2(\square)$, we can extend $\Lambda$ to $H$, and the inequality \eqref{claim1eq} follows in $\square$. And the $\Lambda$ is a continuous linear map on $H$. And there exists $w\in H$ such that for all $\xi \in H$, $\Lambda(\xi)=(\xi,w)_{L^2(\square)}$, by  Riesz representation theorem.\end{proof}

\section*{\textbf{Declarations}}
\noindent
\textbf{Acknowledgments.} 
L. Esquivel and I. Rivas would like to express their gratitude to Universidad del Valle, Department of Mathematics for all the facilities used to realize this work. 

\noindent
\textbf{Ethical approval.} Not applicable.

\noindent
\textbf{Competing interests.} The authors declare that they have no conflict of interest.

\noindent
\textbf{Authors' contributions.}
Both authors have contributed equally to the paper. 

\noindent
\textbf{Availability of data and materials.} Data sharing does not apply to this article, as no data sets were generated or analyzed during the current study.

\end{document}